\documentclass{amsart}
\usepackage{url}
\usepackage{color}
\usepackage{times}
\usepackage{cite}
\usepackage{enumerate,latexsym}
\usepackage{latexsym}
\usepackage{amsmath,amssymb}
\usepackage{amsthm}
\usepackage{verbatim}
\usepackage{mathrsfs}
\newtheorem{thm}{Theorem}[section]
\newtheorem*{theorem*}{Theorem}
\newtheorem*{acknowledgement*}{Acknowledgement}
\newtheorem{cor}[thm]{Corollary}

\newtheorem{lem}[thm]{Lemma}
\newtheorem{prop}[thm]{Proposition}
\theoremstyle{definition}
\newtheorem{defn}[thm]{Definition}
\theoremstyle{remark}
\newtheorem{rem}[thm]{Remark}

\numberwithin{equation}{section}

\newcommand{\set}[1]{\left\{#1\right\}}
\newcommand{\Real}{\mathbb R}

\newcommand{\xX}[0]{\mathbf{x}}

\newcommand{\nN}[0]{\mathbf{n}}

\newcommand{\OO}{\mathbf{0}}

\newcommand{\rstar}[1]{(\star_{#1})}

\title[Smooth compactness for self-expanders]{Smooth compactness for spaces of asymptotically conical self-expanders of mean curvature flow}

\author{Jacob Bernstein}
\address{Department of Mathematics, Johns Hopkins University, 3400 N. Charles Street, Baltimore, MD 21218}
\email{bernstein@math.jhu.edu}

\author{Lu Wang}
\address{Department of Mathematics, University of Wisconsin-Madison, 480 Lincoln Drive, Madison, WI 53706}
\email{luwang@math.wisc.edu}

\thanks{The first author was partially supported by the NSF Grant DMS-1609340. The second author was partially supported by the NSF Grants DMS-1406240 and DMS-1834824, an Alfred P. Sloan Research Fellowship, the office of the Vice Chancellor for Research and Graduate Education at the University of Wisconsin-Madison with funding from the Wisconsin Alumni Research Foundation and a Vilas Early Investigator Award.}

\begin{document}

\begin{abstract}
We show compactness in the locally smooth topology for certain natural families of asymptotically conical self-expanding solutions of mean curvature flow. Specifically, we show such compactness for the set of all two-dimensional self-expanders of a fixed topological type and, in all dimensions, for the set of self-expanders of low entropy and for the set of mean convex self-expanders with strictly mean convex asymptotic cones. From this we deduce that the natural projection map from the space of parameterizations of asymptotically conical self-expanders to the space of parameterizations of the asymptotic cones is proper for these classes.   
\end{abstract}

\maketitle

\section{Introduction} \label{IntroSec}
A \emph{hypersurface}, i.e., a properly embedded codimension-one submanifold, $\Sigma\subset\mathbb{R}^{n+1}$, is a \emph{self-expander} if
\begin{equation} \label{ExpanderEqn}
\mathbf{H}_\Sigma-\frac{\mathbf{x}^\perp}{2}=\mathbf{0}.
\end{equation}
Here 
$$
\mathbf{H}_\Sigma=\Delta_\Sigma\mathbf{x}=-H_\Sigma\mathbf{n}_\Sigma=-\mathrm{div}_\Sigma(\mathbf{n}_\Sigma)\mathbf{n}_\Sigma
$$
is the mean curvature vector, $\mathbf{n}_\Sigma$ is the unit normal, and $\mathbf{x}^\perp$ is the normal component of the position vector. Self-expanders arise naturally in the study of mean curvature flow. Indeed, $\Sigma$ is a self-expander if and only if the family of homothetic hypersurfaces
$$
\left\{\Sigma_t\right\}_{t>0}=\left\{\sqrt{t}\, \Sigma\right\}_{t>0}
$$
is a \emph{mean curvature flow} (MCF), that is, a solution to the flow
$$
\left(\frac{\partial \mathbf{x}}{\partial t}\right)^\perp=\mathbf{H}_{\Sigma_t}.
$$
Self-expanders are expected to model the behavior of a MCF as it emerges from a conical singularity \cite{AIC}. They are also expected to model the long time behavior of the flow \cite{EHAnn}.

Throughout the paper $n,k\geq 2$ are integers and $\alpha\in (0,1)$. Let $\Gamma$ be a $C^{k,\alpha}_{*}$-asymptotically conical $C^{k,\alpha}$-hypersurface in $\mathbb{R}^{n+1}$ and let $\mathcal{L}(\Gamma)$ be the link of the asymptotic cone of $\Gamma$.  For instance, if $\lim_{\rho\to 0^+} \rho \Gamma=\mathcal{C}$ in $C^{k, \alpha}_{loc} (\mathbb{R}^{n+1}\setminus\{\mathbf{0}\})$, where $\mathcal{C}$ is a cone, then $\Gamma$ is $C^{k,\alpha}_{*}$-asymptotically conical with asymptotic cone $\mathcal{C}$. For technical reasons, the actual definition is slightly weaker -- see Section 3 of \cite{BernsteinWangBanach} for the details. We denote the space of $C^{k,\alpha}_{*}$-asymptotically conical $C^{k,\alpha}$-hypersurfaces in $\mathbb{R}^{n+1}$ by $\mathcal{ACH}^{k,\alpha}_n$.

We now introduce the classes we will consider. First, for $g\geq 0$ and $e\geq 1$, let
$$
\mathcal{E}^{k,\alpha}_{\mathrm{top}}({g,e})=\set{\Gamma\in \mathcal{ACH}^{k,\alpha}_2\colon \Gamma \mbox{ satisfies \eqref{ExpanderEqn}  and $\Gamma$ is of genus $g$ with $e$ ends}},
$$
be the space of $C^{k,\alpha}_{*}$-asymptotically conical self-expanders in $\Real^3$ with genus $g$ and $e$ ends. Similarly, for any $h_0>0$, let
$$
\mathcal{E}^{k,\alpha}_{n,\mathrm{mc}}(h_0)=\set{\Gamma\in \mathcal{ACH}^{k,\alpha}_n\colon \Gamma \mbox{ satisfies \eqref{ExpanderEqn}, $H_{\Gamma}>0$,  $H_{\mathcal{L}(\Gamma)}\geq h_0$}},
$$
be the space of $C^{k,\alpha}_{*}$-asymptotically conical self-expanders in $\Real^{n+1}$ which are strictly mean convex and have uniformly strictly mean convex asymptotic cones. Finally, for $1<\Lambda_0<2$, let
$$
\mathcal{E}^{k,\alpha}_{n,\mathrm{ent}}({\Lambda_0})=\set{\Gamma\in \mathcal{ACH}^{k,\alpha}_n\colon \Gamma \mbox{ satisfies \eqref{ExpanderEqn}  and $\lambda[\Gamma]\leq \Lambda_0$}},
$$
be the space of $C^{k,\alpha}_{*}$-asymptotically conical self-expanders in $\Real^{n+1}$ which have entropy less than or equal to $\Lambda_0$. See Section \ref{EntropySec} for the definition of entropy.

We prove the following smooth compactness result for the spaces $\mathcal{E}^{k,\alpha}_{\mathrm{top}}({g,e})$, $\mathcal{E}^{k,\alpha}_{n,\mathrm{mc}}({h_0})$ and, under suitable hypotheses on $\Lambda_0$, $\mathcal{E}^{k,\alpha}_{n,\mathrm{ent}}({\Lambda_0})$.

\begin{thm}\label{SmoothCpctThm}
The following holds:
\begin{enumerate}
\item \label{2DCpctItem} If $\Sigma_i\in \mathcal{E}^{k,\alpha}_{\mathrm{top}}({g,e})$ and $\mathcal{L}(\Sigma_i)\to \sigma$ in $C^{k,\alpha}(\mathbb{S}^2)$, then there is a $\Sigma\in  \mathcal{E}^{k,\alpha}(g,e)$ with $\mathcal{L}(\Sigma)=\sigma$ so that, up to passing to a subsequence, $\Sigma_i\to \Sigma$ in $C^\infty_{loc}(\Real^3)$.
\item \label{MCCpctItem} If $\Sigma_i \in\mathcal{E}^{k,\alpha}_{n,\mathrm{mc}}({h_0})$ and $\mathcal{L}(\Sigma_i)\to \sigma$ in $C^{k,\alpha}(\mathbb{S}^n)$, then there is a $\Sigma\in  \mathcal{E}^{k,\alpha}_{n,\mathrm{mc}}({h_0})$ with $\mathcal{L}(\Sigma)=\sigma$ so that, up to passing to a subsequence, $\Sigma_i\to \Sigma$ in $C^\infty_{loc}(\Real^{n+1})$.
\item \label{EntropyCpctItem} If Assumption \eqref{Assump1} of Section \ref{EntropySec} holds, $\Sigma_i \in\mathcal{E}^{k,\alpha}_{n,\mathrm{ent}}({\Lambda_0})$ for $\Lambda_0< \Lambda<2$ and $\mathcal{L}(\Sigma_i)\to \sigma$ in $C^{k,\alpha}(\mathbb{S}^n)$, then there is a $\Sigma\in  \mathcal{E}^{k,\alpha}_{n,\mathrm{ent}}({\Lambda_0})$ with $\mathcal{L}(\Sigma)=\sigma$ so that, up to passing to a subsequence, $\Sigma_i\to \Sigma$ in $C^\infty_{loc}(\Real^{n+1})$. 
\end{enumerate}
\end{thm}

In \cite{BernsteinWangBanach}, the authors showed that the space $\mathcal{ACE}_n^{k,\alpha}(\Gamma)$ -- see \eqref{ACEDef} below -- of asymptotically conical parameterizations of self-expanders modeled on $\Gamma$ (modulo reparameterizations fixing the parameterization of the asymptotic cone) possesses a natural Banach manifold structure modeled on $C^{k, \alpha}(\mathcal{L}(\Gamma); \mathbb{R}^{n+1})$. They further showed that the map
$$
\Pi\colon \mathcal{ACE}_n^{k,\alpha}(\Gamma)\to C^{k, \alpha}(\mathcal{L}(\Gamma); \mathbb{R}^{n+1})
$$
given by $\Pi([\mathbf{f}])=\mathrm{tr}^1_{\infty}[\mathbf{f}]$ is smooth and Fredholm of index $0$. As such, by work of Smale \cite{Smale}, as long as $\Pi$ is proper it possesses a well-defined mod 2 degree. In fact as shown in \cite{BernsteinWangIntDegree}, when the map $\Pi$ is proper it possesses an integer degree. These results are all analogs of work of White \cite{WhiteEI} who proved such results for a large class of variational problems for parameterizations from compact manifolds -- see also \cite{WhiteVM}.

In general, the map $\Pi\colon \mathcal{ACE}^{k,\alpha}_n(\Gamma)\to C^{k,\alpha}(\mathcal{L}(\Gamma); \mathbb{R}^{n+1})$ is not proper. However, using Theorem \ref{SmoothCpctThm}, we give several natural subsets of $\mathcal{ACE}^{k,\alpha}_n(\Gamma)$ on which the restriction of $\Pi$ is proper. This should be compared to \cite{WhiteMD}. As a first step, it is necessary to shrink the range of $\Pi$. To that end, for any $\Gamma\in\mathcal{ACH}_n^{k,\alpha}$, let
$$
\mathcal{V}_{\mathrm{emb}}^{k,\alpha}(\Gamma)= \set{\varphi\in C^{k,\alpha}(\mathcal{L}(\Gamma);\mathbb{R}^{n+1}) \colon \mbox{$\mathscr{E}^{\mathrm{H}}_1[\varphi]$ is an embedding}},
$$
be the space of parameterizations of embedded cones. Here $\mathscr{E}^{\mathrm{H}}_1[\varphi]$ is the homogeneous degree-one extension of $\varphi$. This is readily seen to be an open subset of $C^{k,\alpha}(\mathcal{L}(\Gamma);\mathbb{R}^{n+1})$. It also follows from the definition of $\mathcal{ACE}^{k,\alpha}_n(\Gamma)$ that $\Pi\colon \mathcal{ACE}^{k,\alpha}_n(\Gamma)\to \mathcal{V}_{\mathrm{emb}}^{k,\alpha}(\Gamma)$.

\begin{thm} \label{2DProperThm}
For any $\Gamma\in\mathcal{ACH}_2^{k,\alpha}$, $\Pi\colon \mathcal{ACE}^{k,\alpha}_2(\Gamma)\to\mathcal{V}_{\mathrm{emb}}^{k,\alpha}(\Gamma)$ is proper.
\end{thm}

\begin{thm} \label{EntropyProperThm}
For $\Gamma\in\mathcal{ACH}_n^{k,\alpha}$ and $\Lambda>1$, let
$$	
\mathcal{V}_{\mathrm{ent}}(\Gamma,\Lambda)=\set{\varphi\in \mathcal{V}_{\mathrm{emb}}^{k,\alpha}(\Gamma)\colon \lambda[\mathscr{E}^{\mathrm{H}}_1[\varphi](\mathcal{C}(\Gamma))]<\Lambda}
$$
and
$$
\mathcal{U}_{\mathrm{ent}}(\Gamma, \Lambda)=\set{[\mathbf{f}]\in \mathcal{ACE}_n^{k,\alpha}(\Gamma)\colon \lambda[\mathbf{f}(\Gamma)]<\Lambda}.
$$
The following is true:
\begin{enumerate}
\item \label{OpenExpanderItem} $\mathcal{U}_{\mathrm{ent}}(\Gamma, \Lambda)$ is an open subset of $\mathcal{ACE}_n^{k,\alpha}(\Gamma)$.
\item \label{OpenConeItem} $\mathcal{V}_{\mathrm{ent}}(\Gamma,\Lambda)$ is an open subset of $C^{k,\alpha}(\mathcal{L}(\Gamma); \Real^{n+1})$.
\item \label{ProperItem} If $\rstar{n, \Lambda}$ holds for $\Lambda<2$, then $\Pi|_{\mathcal{U}_{\mathrm{ent}}(\Gamma, \Lambda)}\colon\mathcal{U}_{\mathrm{ent}}(\Gamma, \Lambda)\to \mathcal{V}_{\mathrm{ent}}(\Gamma, \Lambda)$ is proper.
\end{enumerate}	 
\end{thm}

\begin{thm} \label{MCProperThm}
For $\Gamma\in\mathcal{ACH}_n^{k,\alpha}$, let
$$
\mathcal{V}_{\mathrm{mc}}(\Gamma)=\set{\varphi\in	\mathcal{V}_{\mathrm{emb}}^{k,\alpha}(\Gamma) \colon  H_\sigma>0 \mbox{ where $\sigma=\mathcal{L}[\mathscr{E}^\mathrm{H}_1[\varphi](\mathcal{C}(\Gamma))]$}}
$$
and
$$
\mathcal{U}_{\mathrm{mc}}(\Gamma)=\set{[\mathbf{f}]\in \mathcal{ACE}_n^{k,\alpha}(\Gamma)\colon H_{\mathbf{f}(\Gamma)}>0, \Pi([\mathbf{f}])\in \mathcal{V}_{\mathrm{mc}}(\Gamma)}.
$$
The following is true:
\begin{enumerate}
\item \label{OpenMCExpanderItem} $\mathcal{U}_{\mathrm{mc}}(\Gamma)$ is an open subset of $\mathcal{ACE}_n^{k,\alpha}(\Gamma)$.
\item \label{OpenMCConeItem} $\mathcal{V}_{\mathrm{mc}}(\Gamma)$ is an open subset of $C^{k,\alpha}(\mathcal{L}(\Gamma); \Real^{n+1})$.
\item \label{MCOpenItem} $\Pi|_{\mathcal{U}_{\mathrm{mc}}(\Gamma)}\colon\mathcal{U}_{\mathrm{mc}}(\Gamma)\to \mathcal{V}_{\mathrm{mc}}(\Gamma)$ is a local diffeomorphism.
\item \label{MCProperItem} $\Pi|_{\mathcal{U}_{\mathrm{mc}}(\Gamma)}\colon\mathcal{U}_{\mathrm{mc}}(\Gamma)\to \mathcal{V}_{\mathrm{mc}}(\Gamma)$ is proper.
\end{enumerate}
In particular, for each component $\mathcal{V}^\prime$ of $\mathcal{V}_\mathrm{mc}(\Gamma)$, there is an integer $l^\prime\geq 0$ so $\mathcal{U}^\prime=\Pi^{-1}(\mathcal{V}^\prime)\cap \mathcal{U}_\mathrm{mc}(\Gamma)$ has $l^\prime$ components and for each component $\mathcal{U}^{\prime\prime}$ of $\mathcal{U}^\prime$, $\Pi|_{\mathcal{U}^{\prime\prime}}\colon \mathcal{U}^{\prime\prime}\to \mathcal{V}^\prime$ is a (finite) covering map.
\end{thm}

Finally, as an application of Theorem \ref{MCProperThm} and a result of Huisken \cite{HuiskenMCFSphere}, we have the following existence and uniqueness  result for self-expanders of a given topological type asymptotic to cones that satisfy a natural pinching condition.

\begin{cor}\label{ApplicationCor}
Let $\sigma\subset \mathbb{S}^n$ be a connected, strictly mean convex, $C^{k,\alpha}$-hypersurface. In addition, if $n\geq 3$, suppose that $\sigma$ satisfies
\begin{equation} \label{PinchingEqn}
\left\{
\begin{split}
|A_\sigma|^2 & < \frac{1}{n-2} H_\sigma^2+2,\quad n\geq 4; \\
|A_\sigma|^2 & < \frac{3}{4} H_\sigma^2+\frac{4}{3}, \quad\quad \, \, \, \, n=3.
\end{split} 
\right.
\end{equation}
There exists a smooth self-expander $\Sigma\in \mathcal{ACH}^{k,\alpha}_n$ with $\mathcal{L}(\Sigma)=\sigma$, $H_{\Sigma}>0$ and so $\Sigma$ is diffeomorphic to $\Real^n$. Moreover, if $\pi_0(\mathrm{Diff}^+(\mathbb{S}^{n-1}))=0$, i.e., the group of orientation-preserving diffeomorphisms of $\mathbb{S}^{n-1}$ is path-connected, then $\Sigma$ is the unique self-expander with these properties.
\end{cor}

\begin{rem}
Hypothesis \eqref{PinchingEqn} is required only so that classical mean curvature flow can be used to show the space of admissible $\sigma$ is path-connected.
\end{rem}

\begin{rem}
By work of Cerf \cite{Cerf1, Cerf2} and Smale \cite{SmaleSphere, SmaleCobordism} it is known, for $n\in \set{2,3,4,6}$, that $\pi_0(\mathrm{Diff}^+(\mathbb{S}^{n-1}))=0$. 
\end{rem}

\begin{rem}
Using only the mean convexity condition, a variational argument due to Ilmanen that is sketched in \cite{IlmanenLec} and carried out by Ding in \cite{Ding} gives the existence of a self-expanding solution with link $\sigma$. However, this method cannot directly say anything about the topology of the constructed self-expanders.  
\end{rem}

\section{Notation and background} \label{NotationSec}
In this section we fix notation and also recall the main definitions from \cite{BernsteinWangBanach} we need.  The interested reader should consult Sections 2 and 3 of \cite{BernsteinWangBanach} for specifics and further details.

\subsection{Basic notions} \label{NotionSubsec}
Denote a (open) ball in $\mathbb{R}^n$ of radius $R$ and center $x$ by $B_R^n(x)$ and the closed ball by $\bar{B}^n_R(x)$. We often omit the superscript $n$ when its value is clear from context. We also omit the center when it is the origin. 

For an open set $U\subset\mathbb{R}^{n+1}$, a \emph{hypersurface in $U$}, $\Gamma$, is a smooth, properly embedded, codimension-one  submanifold of $U$. We also consider hypersurfaces of lower regularity and given an integer $k\geq 2$ and $\alpha\in (0,1)$ we define a \emph{$C^{k,\alpha}$-hypersurface in $U$} to be a properly embedded, codimension-one $C^{k,\alpha}$ submanifold of $U$. When needed, we distinguish between a point $p\in\Gamma$ and its \emph{position vector} $\mathbf{x}(p)$.

Consider the hypersurface $\mathbb{S}^n\subset\mathbb{R}^{n+1}$, the unit $n$-sphere in $\mathbb{R}^{n+1}$. A \emph{hypersurface in $\mathbb{S}^n$}, $\sigma$, is a closed, embedded, codimension-one smooth submanifold of $\mathbb{S}^n$ and \emph{$C^{k,\alpha}$-hypersurfaces in $\mathbb{S}^n$} are defined likewise. Observe, that $\sigma$ is a closed codimension-two submanifold of $\mathbb{R}^{n+1}$ and so we may associate to each point $p\in\sigma$ its position vector $\mathbf{x}(p)$. Clearly, $|\mathbf{x}(p)|=1$.

A \emph{cone} is a set $\mathcal{C}\subset\mathbb{R}^{n+1}\setminus\{\mathbf{0}\}$ that is dilation invariant around the origin. That is, $\rho\,\mathcal{C}=\mathcal{C}$ for all $\rho>0$. The \emph{link} of the cone is the set $\mathcal{L}[\mathcal{C}]=\mathcal{C}\cap\mathbb{S}^{n}$. The cone is \emph{regular} if its link is a smooth hypersurface in $\mathbb{S}^{n}$ and \emph{$C^{k,\alpha}$-regular} if its link is a $C^{k,\alpha}$-hypersurface in $\mathbb{S}^n$. For any hypersurface $\sigma\subset\mathbb{S}^n$ the \emph{cone over $\sigma$}, $\mathcal{C}[\sigma]$, is the cone defined by 
$$
\mathcal{C}[\sigma]=\left\{\rho p\colon p\in\sigma, \rho>0\right\}\subset\mathbb{R}^{n+1}\setminus\{\mathbf{0}\}.
$$
Clearly, $\mathcal{L}[\mathcal{C}[\sigma]]=\sigma$. 

\subsection{Function spaces} \label{FunctionSubsec}
Let $\Gamma$ be a properly embedded, $C^{k,\alpha}$ submanifold of an open set $U\subset\mathbb{R}^{n+1}$. There is a natural Riemannian metric, $g_\Gamma$, on $\Gamma$ of class $C^{k-1,\alpha}$ induced from the Euclidean one. As we always take $k\geq 2$, the Christoffel symbols of this metric, in appropriate coordinates, are well-defined and of regularity $C^{k-2,\alpha}$. Let $\nabla_\Gamma$ be the covariant derivative on $\Gamma$. Denote by $d_\Gamma$ the geodesic distance on $\Gamma$ and by $B^\Gamma_R(p)$ the (open) geodesic ball in $\Gamma$ of radius $R$ and center $p\in\Gamma$. For $R$ small enough so that $B_{R}^\Gamma(p)$ is strictly geodesically convex and $q\in B^\Gamma_R(p)$, denote by $\tau^\Gamma_{p,q}$ the parallel transport along the unique minimizing geodesic in $B^\Gamma_R(p)$ from $p$ to $q$. 

Throughout the rest of this subsection, let $\Omega$ be a domain in $\Gamma$, $l$ an integer in $[0,k]$, $\beta\in (0,1)$ and $d\in\mathbb{R}$. Suppose $l+\beta\leq k+\alpha$. We first consider the following norm for functions on $\Omega$:
$$
\Vert f\Vert_{l; \Omega}=\sum_{i=0}^l \sup_{\Omega} |\nabla_\Gamma^i f|.
$$
We then let
$$
C^l(\Omega)=\left\{f\in C_{loc}^l(\Omega)\colon \Vert f\Vert_{l; \Omega}<\infty\right\}.
$$
We next define the H\"{o}lder semi-norms for functions $f$ and tensor fields $T$ on $\Omega$: 
$$
[f]_{\beta; \Omega} =\sup_{\substack{p,q\in\Omega \\ q\in B^\Gamma_{\delta}(p)\setminus\{p\}}} \frac{|f(p)-f(q)|}{d_\Gamma(p,q)^\beta} 
\mbox{ and } 
[T]_{\beta; \Omega} =\sup_{\substack{p,q\in\Omega \\ q\in B^\Gamma_{\delta}(p)\setminus\{p\}}} \frac{|T(p)-(\tau^\Gamma_{p,q})^* T(q)|}{d_\Gamma(p,q)^\beta},
$$
where $\delta=\delta(\Gamma,\Omega)>0$ so that for all $p\in\Omega$, $B^\Gamma_\delta(p)$ is strictly geodesically convex. We further define the norm for functions on $\Omega$:
$$
\Vert f\Vert_{l, \beta; \Omega}=\Vert f\Vert_{l; \Omega}+[\nabla_\Gamma^l f]_{\beta; \Omega},
$$
and let 
$$
C^{l, \beta}(\Omega)=\left\{f\in C_{loc}^{l, \beta}(\Omega)\colon \Vert f\Vert_{l, \beta; \Omega}<\infty\right\}.
$$

We also define the following weighted norms for functions on $\Omega$:
$$
\Vert f\Vert_{l; \Omega}^{(d)}=\sum_{i=0}^l\sup_{p\in\Omega} \left(|\mathbf{x}(p)|+1\right)^{-d+i} |\nabla_\Gamma^i f(p)|.
$$
We then let 
$$
C^{l}_d(\Omega)=\left\{f\in C^l_{loc}(\Omega)\colon \Vert f\Vert_{l; \Omega}^{(d)}<\infty\right\}.
$$
We further define the following weighted H\"{o}lder semi-norms for functions $f$ and tensor fields $T$ on $\Omega$:
\begin{align*}
[f]_{\beta; \Omega}^{(d)} & =\sup_{\substack{p,q\in\Omega \\ q\in B^\Gamma_{\delta_p}(p)\setminus\{p\}}} \left((|\mathbf{x}(p)|+1)^{-d+\beta}+(|\mathbf{x}(q)|+1)^{-d+\beta}\right) \frac{|f(p)-f(q)|}{d_\Gamma(p,q)^\beta}, \mbox{ and}, \\
[T]_{\beta; \Omega}^{(d)} & =\sup_{\substack{p,q\in\Omega \\ q\in B^\Gamma_{\delta_p}(p)\setminus\{p\}}} \left((|\mathbf{x}(p)|+1)^{-d+\beta}+(|\mathbf{x}(q)|+1)^{-d+\beta}\right) \frac{|T(p)-(\tau^\Gamma_{p,q})^* T(q)|}{d_\Gamma(p,q)^\beta},
\end{align*}
where $\eta=\eta(\Omega,\Gamma)\in (0,\frac{1}{4})$ so that for any $p\in\Gamma$, letting $\delta_p=\eta (|\mathbf{x}(p)|+1)$, $B_{\delta_p}^\Gamma(p)$ is strictly geodesically convex. Finally, we define the norm for functions on $\Omega$:
$$
\Vert f\Vert_{l, \beta; \Omega}^{(d)}=\Vert f\Vert_{l; \Omega}^{(d)}+[\nabla_\Gamma^l f]_{\beta; \Omega}^{(d-l)},
$$
and we let
$$
C^{l,\beta}_d(\Omega)=\left\{f\in C^{l,\beta}_{loc}(\Omega)\colon \Vert f\Vert_{l, \beta; \Omega}^{(d)}<\infty\right\}.
$$
We follow the convention that $C^{l,0}_{loc}=C^{l}_{loc}$, $C^{l,0}=C^l$ and $C^{l,0}_d=C^l_d$ and that $C^{0, \beta}_{loc}=C^\beta_{loc}$,  $C^{0,\beta}=C^\beta$ and $C^{0,\beta}_d=C^\beta_d$. The notation for the corresponding norms is abbreviated in the same fashion.

\subsection{Homogeneous functions and homogeneity at infinity} \label{HomoFuncSubsec}
Fix a $C^{k,\alpha}$-regular cone $\mathcal{C}$ with its link $\mathcal{L}$. By our definition $\mathcal{C}$ is a $C^{k,\alpha}$-hypersurface in $\mathbb{R}^{n+1}\setminus\{\mathbf{0}\}$. For $R>0$ let $\mathcal{C}_R=\mathcal{C}\setminus\bar{B}_R$. There is an $\eta=\eta(\mathcal{L},R)>0$ so that for any $p\in\mathcal{C}_R$, $B^{\mathcal{C}}_{\delta_p}(p)$ is strictly geodesically convex, where $\delta_p=\eta(|\mathbf{x}(p)|+1)$. We also fix an integer $l\in [0,k]$ and $\beta\in [0,1)$ with $l+\beta\leq k+\alpha$.

A map $\mathbf{f}\in C^{l,\beta}_{loc}(\mathcal{C}; \mathbb{R}^M)$ is \emph{homogeneous of degree $d$} if $\mathbf{f}(\rho p)=\rho^d \mathbf{f}(p)$ for all $p\in\mathcal{C}$ and $\rho>0$. Given a map $\varphi\in C^{l,\beta}(\mathcal{L}; \mathbb{R}^M)$ the \emph{homogeneous extension of degree $d$} of $\varphi$ is the map $\mathscr{E}_d^{\mathrm{H}}[\varphi]\in C^{l,\beta}_{loc}(\mathcal{C}; \mathbb{R}^M)$ defined by 
$$
\mathscr{E}_d^{\mathrm{H}}[\varphi](p)=|\mathbf{x}(p)|^d \varphi(|\mathbf{x}(p)|^{-1}p).
$$
Conversely, given a homogeneous $\mathbb{R}^M$-valued map of degree $d$, $\mathbf{f}\in C^{l,\beta}_{loc}(\mathcal{C}; \mathbb{R}^M)$, let $\varphi=\mathrm{tr}[\mathbf{f}]\in C^{l,\beta}(\mathcal{L}; \mathbb{R}^M)$, the \emph{trace} of $\mathbf{f}$, be the restriction of $\mathbf{f}$ to $\mathcal{L}$. Clearly, $\mathbf{f}$ is the homogeneous extension of degree $d$ of $\varphi$.

A map $\mathbf{g}\in C^{l,\beta}_{loc}(\mathcal{C}_R; \mathbb{R}^M)$ is \emph{asymptotically homogeneous of degree $d$} if 
$$
\lim_{\rho\to 0^+} \rho^d \mathbf{g}(\rho^{-1}p)=\mathbf{f}(p) \mbox{ in $C^{l,\beta}_{loc}(\mathcal{C}; \mathbb{R}^M)$}
$$
for some $\mathbf{f}\in C^{l,\beta}_{loc}(\mathcal{C}; \mathbb{R}^M)$ that is homogeneous of degree $d$. For such a $\mathbf{g}$ we define the \emph{trace at infinity} of $\mathbf{g}$ by $\mathrm{tr}^d_\infty[\mathbf{g}]=\mathrm{tr}[\mathbf{f}]$. 
We define
$$
C^{l,\beta}_{d,\mathrm{H}}(\mathcal{C}_R; \mathbb{R}^M)=\left\{\mathbf{g}\in C^{l,\beta}_d(\mathcal{C}_R; \mathbb{R}^M)\colon \mbox{$\mathbf{g}$ is asymptotically homogeneous of degree $d$}\right\}.
$$
It is straightforward to verify that $C^{l,\beta}_{d,\mathrm{H}}(\mathcal{C}_R; \mathbb{R}^M)$ is a closed subspace of $C^{l,\beta}_d(\mathcal{C}_R; \mathbb{R}^M)$ and that
$$
\mathrm{tr}^d_\infty\colon C^{l,\beta}_{d,\mathrm{H}}(\mathcal{C}_R; \mathbb{R}^M)\to C^{l,\beta}(\mathcal{L}; \mathbb{R}^M)
$$
is a bounded linear map. Finally, $\mathbf{x}|_{\mathcal{C}_R}\in C^{k,\alpha}_{1,\mathrm{H}}(\mathcal{C}_R; \mathbb{R}^{n+1})$ and $\mathrm{tr}_\infty^1[\mathbf{x}|_{\mathcal{C}_R}]=\mathbf{x}|_{\mathcal{L}}$.

\subsection{Asymptotically conical hypersurfaces} \label{ACHSec}
A $C^{k,\alpha}$-hypersurface, $\Gamma\subset\mathbb{R}^{n+1}$, is \emph{$C^{k,\alpha}_{*}$-asymptotically conical} if there is a $C^{k,\alpha}$-regular cone, $\mathcal{C}\subset\mathbb{R}^{n+1}$, and a homogeneous transverse section, $\mathbf{v}$, on $\mathcal{C}$ such that $\Gamma$, outside some compact set, is given by the $\mathbf{v}$-graph of a function in $C^{k,\alpha}_1\cap C^{k}_{1,0}(\mathcal{C}_R)$ for some $R>1$. Here a transverse section is a regularized version of the unit normal -- see Section 2.4 of \cite{BernsteinWangBanach} for the precise definition. Observe, that by the Arzel\`{a}-Ascoli theorem one has that, for every $\beta\in [0, \alpha)$,
$$
\lim_{\rho\to 0^+} \rho\Gamma = \mathcal{C} \mbox{ in } C^{k, \beta}_{loc}(\mathbb{R}^{n+1}\setminus \{\mathbf{0}\}).
$$
Clearly, the asymptotic cone, $\mathcal{C}$, is uniquely determined by $\Gamma$ and so we denote it by $\mathcal{C}(\Gamma)$. Let $\mathcal{L}(\Gamma)$ denote the link of $\mathcal{C}(\Gamma)$ and, for $R>0$, let $\mathcal{C}_R(\Gamma)=\mathcal{C}(\Gamma)\setminus \bar{B}_R$. Denote the space of $C^{k,\alpha}_{*}$-asymptotically conical $C^{k,\alpha}$-hypersurfaces in $\mathbb{R}^{n+1}$ by $\mathcal{ACH}^{k,\alpha}_n$.

Finally, let $K$ be a compact set of $\Gamma$ and denote by $\Gamma^\prime=\Gamma\setminus K$. By definition, we may choose $K$ large enough so $\pi_{\mathbf{v}}$ -- the projection of a neighborhood of $\mathcal{C}(\Gamma)$ along $\mathbf{v}$ --  restricts to a $C^{k,\alpha}$ diffeomorphism of $\Gamma^\prime$ onto $\mathcal{C}_R(\Gamma)$. Denote its inverse by $\theta_{\mathbf{v}; \Gamma^\prime}$. 

\subsection{Traces at infinity} \label{TraceSubsec}
Fix an element $\Gamma\in\mathcal{ACH}_n^{k,\alpha}$. Let $l$ be an integer in $[0,k]$ and $\beta\in [0,1)$ such that $l+\beta<k+\alpha$. A map $\mathbf{f}\in C_{loc}^{l,\beta}(\Gamma; \mathbb{R}^M)$ is \emph{asymptotically homogeneous of degree $d$} if $\mathbf{f}\circ\theta_{\mathbf{v}; \Gamma^\prime}\in C^{l,\beta}_{d,\mathrm{H}}(\mathcal{C}_R(\Gamma); \mathbb{R}^M)$ where $\mathbf{v}$ is a homogeneous transverse section on $\mathcal{C}(\Gamma)$ and $\Gamma^\prime, \theta_{\mathbf{v}; \Gamma^\prime}$ are introduced in the previous subsection. The \emph{trace at infinity} of $\mathbf{f}$ is then
$$
\mathrm{tr}_\infty^d[\mathbf{f}]=\mathrm{tr}_\infty^d[\mathbf{f}\circ\theta_{\mathbf{v}; \Gamma^\prime}] \in C^{l,\beta}(\mathcal{L}(\Gamma); \mathbb{R}^M).
$$
Whether $\mathbf{f}$ is asymptotically homogeneous of degree $d$ and the definition of $\mathrm{tr}_{\infty}^d$ are independent of the choice of homogeneous transverse sections on $\mathcal{C}(\Gamma)$. Clearly, $\mathbf{x}|_{\Gamma}$ is asymptotically homogeneous of degree $1$ and $\mathrm{tr}_\infty^{1}[\mathbf{x}|_\Gamma]=\mathbf{x}|_{\mathcal{L}(\Gamma)}$.

We next define the space
$$
C_{d,\mathrm{H}}^{l,\beta}(\Gamma; \mathbb{R}^M)=\left\{\mathbf{f}\in C_{d}^{l,\beta}(\Gamma; \mathbb{R}^M)\colon \mbox{$\mathbf{f}$ is asymptotically homogeneous of degree $d$}\right\}.
$$
One can check that $C_{d,\mathrm{H}}^{l,\beta}(\Gamma; \mathbb{R}^M)$ is a closed subspace of $C_d^{l,\beta}(\Gamma; \mathbb{R}^M)$, and the map 
$$
\mathrm{tr}_{\infty}^d\colon C_{d,\mathrm{H}}^{l,\beta}(\Gamma; \mathbb{R}^M)\to C^{l,\beta}(\mathcal{L}(\Gamma); \mathbb{R}^M)
$$
is a bounded linear map. We further define the set $C^{l,\beta}_{d,0}(\Gamma;\mathbb{R}^M)\subset C_{d,\mathrm{H}}^{l,\beta}(\Gamma; \mathbb{R}^M)$ to be the kernel of $\mathrm{tr}_\infty^d$.

\subsection{Asymptotically conical embeddings} \label{ACESubsec}
Fix an element $\Gamma\in \mathcal{ACH}^{k,\alpha}_n$. We define the space of $C^{k,\alpha}_{*}$-asymptotically conical embeddings of $\Gamma$ into $\mathbb{R}^{n+1}$ to be
$$
\mathcal{ACH}^{k,\alpha}_n(\Gamma)=\left\{\mathbf{f}\in C_{1}^{k,\alpha}\cap C^k_{1,\mathrm{H}}(\Gamma; \mathbb{R}^{n+1})\colon\mbox{$\mathbf{f}$ and $\mathscr{E}_{1}^{\mathrm{H}}\circ\mathrm{tr}_\infty^1[\mathbf{f}]$ are embeddings}\right\}.
$$
Clearly, $\mathcal{ACH}^{k,\alpha}_n(\Gamma)$ is an open set of the Banach space $C^{k,\alpha}_{1}\cap C^{k}_{1,\mathrm{H}}(\Gamma; \mathbb{R}^{n+1})$ with the $\Vert\cdot \Vert_{k,\alpha}^{(1)}$ norm. The hypotheses on $\mathbf{f}$, $\mathrm{tr}_\infty^1[\mathbf{f}]\in C^{k,\alpha}(\mathcal{L}(\Gamma);\mathbb{R}^{n+1})$ ensure
$$
\mathcal{C}[\mathbf{f}]=\mathscr{E}_{1}^{\mathrm{H}}\circ\mathrm{tr}_\infty^1[\mathbf{f}]\colon \mathcal{C}(\Gamma)\to \mathbb{R}^{n+1}\setminus\{\mathbf{0}\}
$$
is a $C^{k,\alpha}$ embedding. As this map is homogeneous of degree one, it parameterizes the $C^{k,\alpha}$-regular cone $\mathcal{C}(\mathbf{f}(\Gamma))$ -- see  \cite[Proposition 3.3]{BernsteinWangBanach}.

Finally, we introduce a natural equivalence relation on $\mathcal{ACH}_{n}^{k,\alpha}(\Gamma)$. First, say a $C^{k,\alpha}$ diffeomorphism $\phi\colon\Gamma\to \Gamma$ \emph{fixes infinity} if $\mathbf{x}|_{\Gamma}\circ \phi\in \mathcal{ACH}^{k,\alpha}_n(\Gamma)$ and
$$
\mathrm{tr}^1_{\infty}[\mathbf{x}|_{\Gamma}\circ \phi]=\mathbf{x}|_{\mathcal{L}(\Gamma)}.
$$ 
Two elements $\mathbf{f}, \mathbf{g}\in\mathcal{ACH}_{n}^{k,\alpha}(\Gamma)$ are equivalent, written $\mathbf{f}\sim\mathbf{g}$, provided there is a $C^{k,\alpha}$ diffeomorphism $\phi\colon\Gamma\to\Gamma$ that fixes infinity so that $\mathbf{f}\circ\phi=\mathbf{g}$. Given $\mathbf{f}\in \mathcal{ACH}_{n}^{k,\alpha}(\Gamma)$ let $[\mathbf{f}]$ be the equivalence class of $\mathbf{f}$. Following \cite{BernsteinWangBanach} we define the space
\begin{equation}
\label{ACEDef}
\mathcal{ACE}_n^{k,\alpha}(\Gamma)=\set{[\mathbf{f}]\colon \mathbf{f}\in \mathcal{ACH}^{k,\alpha}_n(\Gamma) \mbox{ and $\mathbf{f}(\Gamma)$ satisfies \eqref{ExpanderEqn}}}.
\end{equation}

\section{Smooth compactness} \label{CpctSec}
In this section we prove Theorem \ref{SmoothCpctThm}. We first prove compactness in the asymptotic region and then treat the three special cases separately.

\subsection{Asymptotic regularity of self-expanders} \label{ACRegSubsec}
Fix a unit vector $\mathbf{e}$, a point $\mathbf{x}_0\in\mathbb{R}^{n+1}$ and $r,h>0$. Let 
$$
C_{\mathbf{e}}(\mathbf{x}_0,r,h)=\set{\mathbf{x}\in\mathbb{R}^{n+1}\colon |(\mathbf{x}-\mathbf{x}_0)\cdot\mathbf{e}|<h, |\mathbf{x}-\mathbf{x}_0|^2<r^2+|(\mathbf{x}-\mathbf{x}_0)\cdot\mathbf{e}|^2}
$$
be the solid open cylinder with axis $\mathbf{e}$ centered at $\mathbf{x}_0$ and of radius $r$ and height $2h$.

\begin{defn}
Suppose that $l\geq 0$ is an integer and $\beta\in [0,1)$. A hypersurface $\Sigma\subset\mathbb{R}^{n+1}$ is a $C^{l,\beta}$ $\mathbf{e}$-graph of size $\delta$ on scale $r$ at $\mathbf{x}_0$ if there is a function $f\colon B^n_r\subset P_{\mathbf{e}}\to\mathbb{R}$ with
$$
\sum_{j=0}^l r^{-1+j} \Vert\nabla^j f\Vert_0+r^{-1+l+\beta} [\nabla^l f]_\beta < \delta,
$$
where $P_\mathbf{e}$ is the $n$-dimensional subspace of $\mathbb{R}^{n+1}$ normal to $\mathbf{e}$ and the last term on the left hand side will be dropped if $\beta=0$, so that 
$$
\Sigma\cap C_{\mathbf{e}}(\mathbf{x}_0, r, \delta r)=\set{\mathbf{x}_0+\mathbf{x}(x)+f(x)\mathbf{e}\colon x\in B_r^n}.
$$
Moreover, a hypersurface $\sigma\subset\mathbb{S}^n$ is a $C^{l,\beta}$ $\mathbf{e}$-graph of size $\delta$ on scale $r$ at $\mathbf{x}_0$ if $\mathcal{C}[\sigma]$ is so. We omit $C^{l,\beta}$ in the above definitions when the hypersurface is of $C^{l,\beta}$ class or when it is clear from context.
\end{defn}

Let us summarize some elementary properties of this concept. 

\begin{prop}\label{SumProp}
Let $l\geq 2$ be an integer and $\beta\in [0,1)$. The following is true:
\begin{enumerate}
\item \label{Sum1Item} If $\Sigma$ is a $C^{l,\beta}$-hypersurface in $\mathbb{R}^{n+1}$, then for every $\delta>0$ and $p\in\Sigma$, there is an $r=r(\Sigma,p,\delta)>0$ so that $\Sigma$ is an $\mathbf{n}_\Sigma(p)$-graph of size $\delta$ on scale $r$ at $p$.
\item \label{Sum2Item} If $\sigma\subset\mathbb{S}^n$ is an $\mathbf{e}$-graph of size $\delta$ on scale $r$ at $\mathbf{x}_0$ and $\rho>0$, then $\mathcal{C}[\sigma]$ is an $\mathbf{e}$-graph of size $\delta$ on scale $\rho r$ at $\rho\mathbf{x}_0$.
\item \label{Sum3Item} Given a $C^{l,\beta}$-hypersurface $\sigma\subset\mathbb{S}^n$ and $\delta>0$, there is an $r=r(\sigma,\delta)>0$ so that $\sigma$ is an $\mathbf{n}_{\mathcal{C}[\sigma]}(p)$-graph of size $\delta$ on scale $r$ at every $p\in\sigma$.
\item \label{Sum4Item} Suppose $\sigma_i\subset\mathbb{S}^n$ converges in $C^{l,\beta}(\mathbb{S}^n)$ to $\sigma$. If $p_i\in\sigma_i\to p\in\sigma$ and $\sigma$ is an $\mathbf{n}_{\mathcal{C}[\sigma]}(p)$-graph of size $\delta$ on scale $2r$ at $p$, then there is an $i_0$ so that for $i\geq i_0$, $\sigma_i$ is an $\mathbf{n}_{\mathcal{C}[\sigma_i]}(p_i)$-graph of size $2\delta$ on scale $r$ at $p_i$.
\end{enumerate}
\end{prop}

The pseudo-locality property of mean curvature flow gives certain asymptotic regularity for self-expanders that are weakly asymptotic to a cone. 

\begin{prop} \label{AsympRegProp}
Let $l\geq 2$ be an integer and $\beta\in [0,1)$. For each $\delta>0$ and $r>0$ there exist constants $\mathcal{R},M,\gamma,\eta>0$, depending only $n,l,\beta,\delta$ and $r$, so that if $\Sigma$ is a self-expander in $\mathbb{R}^{n+1}$ with 
$$
\lim_{\rho\to 0^+} \mathcal{H}^n\lfloor (\rho\Sigma)=\mathcal{H}^n\lfloor\mathcal{C}[\sigma]
$$
for $\sigma$ a $C^{l,\beta}$-hypersurface in $\mathbb{S}^n$, and $\sigma$ is an $\mathbf{n}_{\mathcal{C}[\sigma]}(p)$-graph of size $\delta$ on scale $r$ at every $p\in\sigma$, then 
\begin{enumerate}
\item \label{LocGraphItem} $\Sigma$ is a $C^{l,\beta}$ $\mathbf{n}_{\mathcal{C}[\sigma]}(p)$-graph of size $\gamma$ on scale $\eta |\mathbf{x}(p)|$ at every $p\in\mathcal{C}[\sigma]\setminus\bar{B}_{\mathcal{R}}$.
\item \label{LinearDecayItem} There is a function $f\colon\mathcal{C}(\Sigma)\setminus\bar{B}_{\mathcal{R}}\to\mathbb{R}$ satisfying 
$$
|f(p)|+|\nabla_{\mathcal{C}[\sigma]} f(p)| \leq M|\mathbf{x}(p)|^{-1}
$$
and so 
$$
\Sigma\setminus \bar{B}_{2\mathcal{R}}=\set{\mathbf{x}(p)+f(p)\mathbf{n}_{\mathcal{C}[\sigma]}(p)\colon p\in\mathcal{C}[\sigma]\setminus\bar{B}_{\mathcal{R}}}\setminus\bar{B}_{2\mathcal{R}}.
$$
\end{enumerate}
\end{prop}

\begin{proof}
For simplicity, we set $\mathcal{C}=\mathcal{C}[\sigma]$. Consider the mean curvature flow (thought of as a space-time track)
$$
\mathcal{S}=\bigcup_{t>0} \left(\sqrt{t}\, \Sigma\right)\times\{t\}.
$$
Let $\bar{\mathcal{S}}=\mathcal{S}\cup(\mathcal{C}\times\{0\})$ so, by our hypothesis, $\bar{\mathcal{S}}$ is an integer Brakke flow (cf. \cite[\S 2]{IlmanenSing} and \cite[\S 6]{IlmanenElliptic}).
Denote by $\bar{\mathcal{S}}_t$ the time $t$ slice of $\bar{\mathcal{S}}$. For Item \eqref{LocGraphItem}, it is sufficient to prove that there are constants $\tau,\gamma,\eta>0$, depending only on $n,l,\beta,\delta$ and $r$, so that for all $t\in [0,\tau]$, $\bar{\mathcal{S}}_t$ is a $C^{l,\beta}$ $\mathbf{n}_{\mathcal{C}}(p)$-graph of size $\gamma$ on scale $\eta$ at every $p\in\sigma$.

By the pseudo-locality property for mean curvature flow (cf. \cite[Theorem 1.5 and Remarks 1.6]{IlmanenNevesSchulze}) there is an $\epsilon\in (0,1)$, depending on $n,\delta$ and $r$, so that for every $t\in [0,16\epsilon^2)$ and $p\in\sigma$, $\bar{\mathcal{S}}_t\cap C_{\mathbf{n}_{\mathcal{C}}(p)}(p,4\epsilon,4\epsilon)$ is the graph of a function $\psi_p(t,x)$ over $T_p\mathcal{C}$ with
$$
(4\epsilon)^{-1}\Vert\psi_p(t,\cdot)\Vert_0+\Vert D_x\psi_p(t,\cdot)\Vert_0\leq 1.
$$
Moreover, as $\psi_p(t,x)$ satisfies
$$
\frac{\partial\psi_p}{\partial t}=\sqrt{1+|D_x\psi_p|^2}\, \mathrm{div}\left(\frac{D_x\psi_p}{\sqrt{1+|D_x\psi_p|^2}}\right),
$$
it follows from the H\"{o}lder estimate for quasi-parabolic equations (cf. \cite[Theorem 1.1 of Chapter 6]{LSU}) that for every $\alpha^\prime\in (0,1)$,
$$
\sup_{t\in [0,4\epsilon^2]} [D_x\psi_p(t,\cdot)]_{\alpha^\prime; B^n_{2\epsilon}}+\sup_{x\in B^n_{2\epsilon}} [D_x\psi_p(\cdot,x)]_{\frac{\alpha^\prime}{2}; [0,4\epsilon^2]} \leq C(n,\alpha^\prime,\epsilon).
$$
Furthermore, we appeal to the estimates of fundamental solutions and the Schauder theory (cf. \cite[(13.1) and Theorem 5.1]{LSU}) to get that
\begin{equation} \label{HolderEstEqn}
\sup_{t\in [0,\epsilon^2]}\Vert \psi_p(t,\cdot)\Vert_{l,\beta; B^n_\epsilon}\leq C^\prime (n,l,\beta,\epsilon).
\end{equation}

Using the equation of $\psi_p$ and the fact that $\psi_p(0,0)=|D_x\psi_p(0,0)|=0$, it follows from \eqref{HolderEstEqn} that
\begin{equation} \label{C1EstEqn}
|\psi_p(t,x)| \leq \tilde{C} \left(|x|^2+t\right) \mbox{ and } |D_x\psi_p(t,x)| \leq \tilde{C}\left(|x|+\sqrt{t}\right),
\end{equation}
where $\tilde{C}=\tilde{C}(n,C,C^\prime)>C^\prime$. In particular, for $\rho\in (0,1)$ and for every $t\in [0,\rho^2\epsilon^2]$,
$$
(\rho\epsilon)^{-1} \Vert\psi_{p}(t,\cdot)\Vert_{0; B^n_{\rho\epsilon}}+\Vert D_x\psi_p(t,\cdot)\Vert_{0; B^n_{\rho\epsilon}} \leq 4\tilde{C}\rho\epsilon.
$$
This together with \eqref{HolderEstEqn} further gives
$$
\sum_{j=0}^l (\rho\epsilon)^{j-1}\Vert D_x^j \psi_p(t,\cdot)\Vert_{0;B^n_{\rho\epsilon}}+(\rho\epsilon)^{l+\beta-1} [D_x^l \psi_p(t,\cdot)]_{\beta; B^n_{\rho\epsilon}} \leq 5\tilde{C}\rho\epsilon.
$$
Now we choose $\rho=(5\tilde{C}\epsilon)^{-1/2}$ so $5\tilde{C}\rho\epsilon<4\rho^{-1}$. Hence, for each $t\in [0,\rho^2\epsilon^2]$, $\bar{\mathcal{S}}_t$ is a $C^{l,\beta}$ $\mathbf{n}_{\mathcal{C}}(p)$-graph of size $4\rho^{-1}$ on scale $\rho\epsilon$ at every $p\in\sigma$. The claim follows immediately with $\tau=\rho^2\epsilon^2,\gamma=4\rho^{-1}$ and $\eta=\rho\epsilon$.

As $\bar{\mathcal{S}}$ is a mean curvature flow away from $(\mathbf{0},0)$, by comparing with shrinking spheres, one observes that
$$
\Sigma\setminus\bar{B}_{2R} \subset \bigcup_{p\in\mathcal{C}\setminus\bar{B}_R} C_{\mathbf{n}_{\mathcal{C}}(p)} (p,\eta |\mathbf{x}(p)|, \gamma\eta |\mathbf{x}(p)|)
$$
for some $R>\tau^{-1/2}$ depending on $n,l,\beta,\delta$ and $r$. Thus, invoking estimate \eqref{C1EstEqn} gives that for every $q\in\Sigma\setminus\bar{B}_{2R}$, 
$$
|\mathbf{x}(q)-\mathbf{x}(\pi(q))|+|\mathbf{n}_\Sigma(q)-\mathbf{n}_{\mathcal{C}}(\pi(q))| \leq \hat{C}(n,\tilde{C}) |\mathbf{x}(q)|^{-1}.
$$
Here $\pi$ is the nearest point projection from $\Sigma\setminus\bar{B}_{2R}$ onto $\mathcal{C}$. Hence, Item \eqref{LinearDecayItem} follows easily from this estimate and the implicit function theorem.
\end{proof}

\begin{cor} \label{EndCompactCor}
If $\Sigma_i\in\mathcal{ACH}^{k,\alpha}_n$ are self-expanders and there is a $C^{k,\alpha}$-hypersurface in $\mathbb{S}^n$, $\sigma$, so $\mathcal{L}(\Sigma_i)\to\sigma$ in $C^{2}(\mathbb{S}^n)$, then there is an $\mathcal{R}^\prime=\mathcal{R}^\prime(n,k,\alpha,\sigma)>0$ and a $C^{k,\alpha}_*$-asymptotically conical self-expanding end $\Sigma$ in $\mathbb{R}^{n+1}\setminus\bar{B}_{\mathcal{R}^\prime}$ with $\mathcal{L}(\Sigma)=\sigma$ so that, up to passing to a subsequence, 
$$
\Sigma_i\to\Sigma \mbox{ in $C^{\infty}_{loc}(\mathbb{R}^{n+1}\setminus\bar{B}_{\mathcal{R}^\prime})$}.
$$
\end{cor}

\begin{proof}
By Items \eqref{Sum3Item} and \eqref{Sum4Item} of Proposition \ref{SumProp}, there are $\delta, r>0$ and $i_0$ so that for $i\geq i_0$ each $\mathcal{L}(\Sigma_i)$ is an $\mathbf{n}_{\mathcal{C}[\sigma]}(p)$-graph of size $\delta$ on scale $r$ for all $p\in \mathcal{L}(\Sigma_i)$. In particular, we may apply Item \eqref{LinearDecayItem} of Proposition \ref{AsympRegProp} using these constants to obtain an $\mathcal{R}$ so that one has uniform graphical estimates for the $\Sigma_i$ in $\Real^{n+1}\setminus \bar{B}_{\mathcal{R}}$. It then follows from the Arzel\`{a}-Ascoli theorem and standard elliptic estimates (see \cite[Theorems 6.17 and 8.24]{GT}) that there is a self-expanding end, $\Sigma$, in $\Real^{n+1}\setminus \bar{B}_{\mathcal{R}}$ so that up to passing to a subsequence $\Sigma_i\to \Sigma$ in $C^\infty_{loc}(\Real^{n+1}\setminus \bar{B}_{\mathcal{R}})$.  In fact, by Item \eqref{LinearDecayItem} of Proposition \ref{AsympRegProp} applied to the $\Sigma_i$ together with the Arzel\`{a}-Ascoli theorem, $\Sigma$ is $C^{0,1}_*$-asymptotic to $\mathcal{C}[\sigma]$. Combining this fact with Item \eqref{LocGraphItem} of Proposition \ref{AsympRegProp} applied to $\Sigma$ gives that $\Sigma$ is actually $C^{k, \alpha}_*$-asymptotic to $\mathcal{C}[\sigma]$ which completes the proof.
\end{proof}

\subsection{Entropy and smooth compactness of $\mathcal{E}_{n, \mathrm{ent}}^{k,\alpha}(\Lambda_0)$} \label{EntropySec}
We recall the notion of entropy introduced by Colding and Minicozzi \cite{CMGenMCF} and use it to prove Item \eqref{EntropyCpctItem} of Theorem \ref{SmoothCpctThm} as well as introduce several auxiliary results needed in other parts of the article.  

First of all, for a hypersurface $\Sigma\subset\mathbb{R}^{n+1}$ the \emph{Gaussian surface area} of $\Sigma$ is 
$$
F[\Sigma]=(4\pi)^{-\frac{n}{2}} \int_\Sigma e^{-\frac{|\mathbf{x}|^2}{4}}\, d\mathcal{H}^n.
$$
Colding and Minicozzi \cite{CMGenMCF} introduced the following notion of \emph{entropy of a hypersurface}:
$$
\lambda[\Sigma]=\sup_{\rho>0,\mathbf{y}\in\mathbb{R}^{n+1}} F[\rho\Sigma+\mathbf{y}].
$$
By identifying $\Sigma$ with $\mathcal{H}^n\lfloor\Sigma$, one extends $F$ and $\lambda$ in an obvious manner to Radon measures on $\mathbb{R}^{n+1}$.  We first record a number of simple observations about the entropy of asymptotically conical self-expanders.

\begin{lem}\label{ExpanderConeEntropyLem}
If $\Sigma\in \mathcal{ACH}^{k,\alpha}_n$ is a self-expander, then
$$
\lambda[\Sigma]=\lambda[\mathcal{C}(\Sigma)].
$$
\end{lem}
\begin{proof}
On the one hand, for $\Sigma\in\mathcal{ACH}^{k,\alpha}_n$,
$$
\lim_{\rho\to 0^+} \rho\Sigma=\mathcal{C}(\Sigma) \mbox{ in $C^{k}_{loc}(\mathbb{R}^{n+1}\setminus\set{\mathbf{0}})$}.
$$
In particular,
$$
\lim_{\rho\to 0^+}\mathcal{H}^n\lfloor (\rho\Sigma)=\mathcal{H}^n\lfloor\mathcal{C}(\Sigma).
$$
By the lower semicontinuity and scaling invariance of entropy
$$
\lambda[\Sigma]\geq\lambda[\mathcal{C}(\Sigma)].
$$

On the other hand, we define $\mathcal{M}=\{\mu_t\}_{t\ge 0}$ to be a family of Radon measures on $\mathbb{R}^{n+1}$ given by
$$
\mu_t=\left\{
\begin{array}{cc}
\mathcal{H}^n\lfloor (\sqrt{t}\, \Sigma) & \mbox{if $t>0$}  \\
\mathcal{H}^n\lfloor \mathcal{C}(\Sigma) & \mbox{if $t=0$},
\end{array}
\right.
$$
so $\mathcal{M}$ is an integer Brakke flow (see \cite[\S 2]{IlmanenSing} and \cite[\S 6]{IlmanenElliptic}). Thus, the Huisken monotonicity formula \cite{Huisken} (see also \cite[Lemma 7]{IlmanenSing}) implies 
$$
\lambda[\Sigma] \leq \lambda[\mathcal{C}(\Sigma)],
$$
which finishes the proof.
\end{proof}

\begin{lem}\label{AreaRatioLem}
There is a constant $\tilde{M}=\tilde{M}(n)$ so that if $\Sigma\in\mathcal{ACH}^{k,\alpha}_n$ is a self-expander, then, for any $R>0$,
$$
\mathcal{H}^n(\Sigma\cap B_R) \leq \tilde{M} \lambda[\mathcal{C}(\Sigma)]R^{n} .
$$
\end{lem}

\begin{proof}
One computes,
$$
R^{-n}\mathcal{H}^n (\Sigma \cap B_R)=\mathcal{H}^n\left((R^{-1} \Sigma)\cap B_1\right)\leq \tilde{M}(n) F[R^{-1} \Sigma]\leq \tilde{M}(n) \lambda[\Sigma]
$$
and so the claim follows from Lemma \ref{ExpanderConeEntropyLem}.
\end{proof}

\begin{lem}\label{EntropyUnifLem}
Fix any $\epsilon>0$. If $\mathcal{C}\subset\mathbb{R}^{n+1}$ is a $C^2$-regular cone and $\mathcal{L}[\mathcal{C}]$ is an $\mathbf{n}_{\mathcal{C}}(p)$-graph of size $\delta$ on scale $r$ at every $p\in\mathcal{L}[\mathcal{C}]$, then there is an $\tilde{\mathcal{R}}=\tilde{\mathcal{R}}(n,\epsilon,\delta,r)$ so that either $\lambda[\mathcal{C}]\leq 1+\epsilon$ or $\lambda[\mathcal{C}]=F[\mathcal{C}+\mathbf{x}_0]$ for some $\mathbf{x}_0\in \bar{B}_{\tilde{\mathcal{R}}}$.
\end{lem}

\begin{proof}
First observe, that as $\mathcal{C}$ is invariant by homotheties one has
$$
\lambda[\mathcal{C}]=\sup_{\mathbf{x}\in \Real^{n+1}} F[\mathcal{C}+\mathbf{x}].
$$
Next observe that an elementary covering argument gives an $A=A(n,\delta,r)$ so that
$$
\mathcal{H}^{n-1}(\mathcal{L}[\mathcal{C}])\leq A.
$$
Hence, there is an $A^\prime=A^\prime(n,\delta,r)>0$ so for all $R>0$ and $\mathbf{x}\in \Real^{n+1}$,
$$
\mathcal{H}^{n}(\mathcal{C}\cap B_R(\mathbf{x}))\leq A^\prime R^n.
$$
A consequence of this is that there is an $\tilde{\mathcal{R}}=\tilde{\mathcal{R}}(n, \epsilon, \delta, r)$ so that for all $\mathbf{x}\notin \bar{B}_{\tilde{\mathcal{R}}}$,
$$
F[\mathcal{C}+\mathbf{x}]\leq 1+\epsilon.
$$
The claim follows from this.
\end{proof}
 
\begin{lem} \label{ContEntropyConeLem}
Let $\sigma_i$ and $\sigma$ be $C^2$-hypersurfaces in $\mathbb{S}^n$. If $\sigma_i\to\sigma$ in $C^2(\mathbb{S}^n)$, then 
$$
\lambda[\mathcal{C}[\sigma_i]]\to\lambda[\mathcal{C}[\sigma]].
$$
\end{lem} 

\begin{proof}
As $\sigma_i\to\sigma$ in $C^2(\mathbb{S}^n)$,
$$
\mathcal{H}^n\lfloor\mathcal{C}[\sigma_i]\to\mathcal{H}^n\lfloor\mathcal{C}[\sigma].
$$
By the lower semicontinuity of entropy
$$
\liminf_{i\to\infty}\lambda[\mathcal{C}[\sigma_i]] \geq \lambda[\mathcal{C}[\sigma]].
$$

By Lemma \ref{EntropyUnifLem}, given $\epsilon>0$ there is an $R=R(n,\epsilon,\sigma)>0$ so that for $i$ sufficiently large, either $\lambda[\mathcal{C}[\sigma_i]]\leq 1+\epsilon$ or $\lambda[\mathcal{C}[\sigma_i]]=F[\mathcal{C}[\sigma_i]+\mathbf{x}_i]$ for some $\mathbf{x}_i\in\bar{B}_{R}$. Observe, that there is an $A=A(n,\sigma)$ so that for $i$ sufficiently large and for every $r>0$ and $\mathbf{x}\in\mathbb{R}^{n+1}$,
$$
\mathcal{H}^n(\mathcal{C}[\sigma_i]\cap B_r(\mathbf{x})) \leq Ar^n.
$$
Hence, we get
$$
\limsup_{i\to\infty} \lambda[\mathcal{C}[\sigma_i]]\leq\max\set{1+\epsilon,\lambda[\mathcal{C}[\sigma]]}.
$$
Passing $\epsilon$ to $0$, as $\lambda[\mathcal{C}[\sigma]]\geq 1$, it follows that
$$
\limsup_{i\to\infty} \lambda[\mathcal{C}[\sigma_i]] \leq\lambda[\mathcal{C}[\sigma]],
$$
completing the proof.
\end{proof}

We are now ready to prove that $\mathcal{E}_{n,\mathrm{ent}}^{k,\alpha}(\Lambda_0)$ is compact. In order to do so we introduce a necessary hypothesis about the entropy of minimal cones. Let $\mathcal{RMC}_n$ denote the space of \emph{regular minimal cones} in $\Real^{n+1}$, that is $\mathcal{C}\in \mathcal{RMC}_n$ if and only if it is a proper subset of $\Real^{n+1}$ and $\mathcal{C}\setminus\set{\OO}$ is a hypersurface in $\Real^{n+1}\setminus\set{\OO}$ that is invariant under dilation about $\OO$ and with vanishing mean curvature. Let $\mathcal{RMC}_n^*$ denote the set of non-flat elements of $\mathcal{RMC}_n$ -- i.e., cones with non-zero curvature somewhere. For any $\Lambda>0$, let 
$$
\mathcal{RMC}_n(\Lambda)=\set{\mathcal{C}\in \mathcal{RMC}_n\colon \lambda[\mathcal{C}]< \Lambda} \mbox{ and } \mathcal{RMC}_n^*(\Lambda)=\mathcal{RMC}^*_n \cap \mathcal{RMC}_n(\Lambda).
$$
Now fix a dimension $n\geq 2$ and a value $ \Lambda>1$.  Consider the following hypothesis:
\begin{equation} \label{Assump1}
\mbox{For all $2\leq l\leq n$, }\mathcal{RMC}_{l}^*(\Lambda)=\emptyset \tag{$\star_{n,\Lambda}$}.
\end{equation}
Observe that all regular minimal cones in $\mathbb{R}^2$ consist of unions of rays and so $\mathcal{RMC}^*_1=\emptyset$. Likewise, as great circles are the only geodesics in $\mathbb{S}^2$, $\mathcal{RMC}_2^*=\emptyset$ and so $\rstar{2,\Lambda}$ always holds. As a consequence of Allard's regularity theorem and a dimension reduction argument, there is always some $\Lambda_n>1$ so that $\rstar{n,\Lambda_n}$ holds.

\begin{proof}[Proof of Item \eqref{EntropyCpctItem} of Theorem \ref{SmoothCpctThm}]
First, by Corollary \ref{EndCompactCor}, there is an $R>0$ so that, up to passing to a subsequence, $\Sigma_i\setminus\bar{B}_R$ converges in $C^\infty_{loc}(\mathbb{R}^{n+1}\setminus\bar{B}_R)$ to some hypersurface $\Sigma^\prime$ in $\mathbb{R}^{n+1}\setminus\bar{B}_R$. Moreover, $\Sigma^\prime$ is a $C^{k,\alpha}_*$-asymptotically conical self-expander in $\mathbb{R}^{n+1}\setminus \bar{B}_R$ and $\mathcal{L}(\Sigma^\prime)=\sigma$.

As $\lambda[\Sigma_i]\leq \Lambda_0<\Lambda$, Lemma \ref{AreaRatioLem} and the standard compactness results, imply that, up to passing to a further subsequence, $\Sigma_i\cap B_{2R}$ converges in the sense of measures to some integral varifold, $V$, in $B_{2R}$. As such, there is an integral varifold, $\Sigma$, in $\Real^{n+1}$ that agrees with $V$ in $B_{2R}$ and $\Sigma^\prime$ outside $\bar{B}_R$.  In particular, $\Sigma$ is smooth and properly embedded in $\Real^{n+1}\setminus \bar{B}_R$. The lower semicontinuity of entropy gives $\lambda[\Sigma]\leq \Lambda_0<\Lambda$, and so it follows from \eqref{Assump1}, a dimension reduction theorem \cite[Theorem 4]{WhiteStratification} and Allard's regularity theorem (e.g., \cite[Theorem 24.2]{Simon}) that $\Sigma$ is actually smooth and properly embedded in $\Real^{n+1}$.  That is, $\Sigma\in \mathcal{E}_{n, \mathrm{ent}}^{k,\alpha}(\Lambda_0)$.  Finally, a further consequence of Allard's regularity theorem \cite{WhiteReg} is that $\Sigma_i\to\Sigma$ in $C^{\infty}_{loc}(\mathbb{R}^{n+1})$, finishing the proof.
\end{proof}

\subsection{Smooth compactness of $\mathcal{E}^{k,\alpha}_{\mathrm{top}}(g,e)$}
Combining the asymptotic compactness property, the area estimates from Lemma \ref{AreaRatioLem} and a result of White  \cite[Theorem 3 (4)]{WhiteInventiones} we can easily prove Item \eqref{2DCpctItem} of Theorem \ref{SmoothCpctThm}.

\begin{proof}[Proof of Item \eqref{2DCpctItem} of Theorem \ref{SmoothCpctThm}]
By Corollary \ref{EndCompactCor} there is an $R>0$ so that, up to passing to a subsequence, 
$$
\Sigma_i\setminus\bar{B}_R\to\Sigma^\prime \mbox{ in $C^{\infty}_{loc}(\mathbb{R}^3\setminus\bar{B}_R)$},
$$
where $\Sigma^\prime$ is a $C^{k,\alpha}_*$-asymptotically conical self-expander in $\mathbb{R}^3\setminus\bar{B}_R$, $\Sigma^\prime$ consists of $e$ annuli and $\mathcal{L}(\Sigma^\prime)=\sigma$. In particular, for $r>R$ sufficiently large, $\partial B_r$ meets $\Sigma$ transversally and also meets each $\Sigma_i$ transversally. As such, $\Sigma_i \cap B_r$ has genus $g$ and $\Sigma_i\cap \partial B_r $ has $e$ components. Thus, it follows that there is a constant $C=C(g,e)$ so
$$
\int_{\partial B_r\cap\Sigma_i} \kappa < C
$$
where $\kappa$ denotes the geodesic curvature of the boundary curve. Moreover, by Lemma \ref{ContEntropyConeLem}, for $i$ sufficiently large, $\lambda[\mathcal{C}(\Sigma_i)] \leq \lambda[\mathcal{C}[\sigma]]+1$ and so, by Lemma \ref{AreaRatioLem}, there is a $C'$ so
$$
\mathcal{H}^2(\Sigma_i\cap B_r) \leq C^\prime.
$$
Hence, it follows from \cite[Theorem 3 (4)]{WhiteInventiones} that, up to passing to a further subsequence, 
$$
\Sigma_i\cap B_r\to\Sigma^{\prime\prime} \mbox{ in $C^{\infty}_{loc}(B_r)$}
$$
where $\Sigma^{\prime\prime}$ is a self-expander in $B_r$ and the convergence is with multiplicity one.   It is clear that $\Sigma^\prime=\Sigma^{\prime\prime}$ in $B_r\setminus\bar{B}_R$. Therefore the result follows with $\Sigma=\Sigma^\prime\cup\Sigma^{\prime\prime}$.
\end{proof}

\subsection{Smooth compactness of $\mathcal{E}^{k,\alpha}_{n,\mathrm{mc}}(h_0)$}
We combine Lemma \ref{AreaRatioLem} with a curvature estimate for mean convex self-expanders in order to prove Item \eqref{MCCpctItem} of Theorem \ref{SmoothCpctThm}.

First we show a curvature estimate for strictly mean convex asympotitically conical self-expanders with strictly mean convex link. Our argument uses the maximum principle and is completely analogous to the one used in \cite{CMGenMCF} for mean convex self-shrinkers.

\begin{lem}\label{CurvBoundLem}
Let $\Sigma\subset \Real^{n+1}$, be an asymptotically conical self-expander. If $\Sigma$ is strictly mean convex and whose asymptotic cone, $\mathcal{C}(\Sigma)$ is $C^2$-regular and strictly mean convex, then, for $p\in\Sigma$,
$$
|A_\Sigma(p)|^2 \leq  K |\mathbf{x}(p)|^2
$$
where
$$
K=\frac{1}{4} \sup_{\mathcal{C}(\Sigma)} \frac{|A_{\mathcal{C}(\Sigma)}|^2}{H_{\mathcal{C}(\Sigma)}^2}=\frac{1}{4} \sup_{\mathcal{L}(\Sigma)} \frac{|A_{\mathcal{L}(\Sigma)}|^2}{H_{\mathcal{L}(\Sigma)}^2}<\infty.
$$
\end{lem}

\begin{proof}
By definition of asymptotic cones and scaling invariance
$$
\lim_{R\to \infty} \sup_{\partial B_R \cap \Sigma}  \frac{|A_{\Sigma}|^2}{H_{\Sigma}^2}=4K.
$$
Hence, for every $\epsilon>0$, there is an $R_\epsilon$ so if $R>R_\epsilon$, then
$$
\sup_{\partial B_R \cap \Sigma}  \frac{|A_{\Sigma}|^2}{H_{\Sigma}^2}\leq 4K+\epsilon.
$$
Computing as in \cite{CMGenMCF} one has
$$
L_\Sigma H_\Sigma =\left(\Delta_\Sigma +\frac{\xX}{2}\cdot \nabla_\Sigma+|A_\Sigma|^2 -\frac{1}{2}\right) H_\Sigma=-H_{\Sigma}
$$
which follows from a variant on Simons' identity, for the rough drift Laplacian,
$$
\left(L_\Sigma A_\Sigma\right)_{ij}=-(A_\Sigma)_{ij}.
$$
For details see the computations in (10.10)-(10.12) and (5.7)-(5.8) of \cite{CMGenMCF}. These are carried out for self-shrinkers, however the formulas for self-expanders follow with a simple sign change.
	
Now consider the new operator 
$$
\mathscr{L}_{H^2_\Sigma}= \Delta_\Sigma+\frac{\xX}{2}\cdot \nabla_\Sigma+2(\nabla_\Sigma \log H_{\Sigma}) \cdot \nabla_\Sigma
$$
We compute that
$$
\left(\mathscr{L}_{H^2_\Sigma} \frac{A_\Sigma}{H_\Sigma}\right)_{ij}= 0
$$
and so
$$
\mathscr{L}_{H^2_\Sigma}\frac{|A_\Sigma|^2}{H_\Sigma^2}=2\left| \nabla_\Sigma \frac{A_\Sigma}{H_\Sigma}\right|^2\geq 0.
$$
Hence, by the maximum principle, for any $R>R_\epsilon$,
$$
\sup_{B_{R} \cap \Sigma}\frac{|A_\Sigma|^2}{H_\Sigma^2} \leq 4K+\epsilon.
$$
Hence, letting $\epsilon\to 0$, gives
$$
\sup_{\Sigma} \frac{|A_\Sigma|^2}{H_\Sigma^2} \leq 4K.
$$
That is, for any $p\in \Sigma$,
$$
|A_{\Sigma}(p)|^2 \leq 4K H_{\Sigma}^2(p)\leq K (\xX(p)\cdot \nN_\Sigma(p))^2\leq K |\xX(p)|^2.
$$
This proves the claim.
\end{proof}

We are now ready to complete the proof of Theorem \ref{SmoothCpctThm}.
\begin{proof}[Proof of Item \eqref{MCCpctItem} of Theorem \ref{SmoothCpctThm}]
By Corollary \ref{EndCompactCor} there is an $R>0$ so that, up to passing to a subsequence, 
$$
\Sigma_i\setminus\bar{B}_R\to\Sigma^\prime \mbox{ in $C^{\infty}_{loc}(\mathbb{R}^{n+1}\setminus\bar{B}_R)$},
$$
where $\Sigma^\prime$ is a $C^{k,\alpha}_*$-asymptotically conical self-expander in $\mathbb{R}^{n+1}\setminus\bar{B}_R$ and $\mathcal{L}(\Sigma^\prime)=\sigma$. The nature of the convergence, ensures that $H_{\sigma}\geq h_0$. Moreover, by Lemma \ref{ContEntropyConeLem}, for $i$ sufficiently large, $\lambda[\mathcal{C}(\Sigma_i)] \leq \lambda[\mathcal{C}[\sigma]]+1$. Thus, by Lemma \ref{AreaRatioLem}, there is a uniform $C^\prime$ so
$$
\mathcal{H}^n(\Sigma_i\cap B_r) \leq C^\prime.
$$

As $\sigma$ is strictly mean convex, for $i$ sufficiently large each $\mathcal{L}(\Sigma_i)$ is strictly mean convex. Indeed, setting
$$
K=\sup_{\sigma} \frac{|A_{\sigma}|^2}{H_{\sigma}^2} \in (0, \infty)
$$
one has, after possibly throwing out a finite sequence of the $\Sigma_i$, that
$$
\sup_{\mathcal{L}(\Sigma_i)} \frac{|A_{\mathcal{L}(\Sigma_i)}|^2}{H_{{\mathcal{L}(\Sigma_i)}}^2} \leq 4 K.
$$
Hence, by Lemma \ref{CurvBoundLem},
$$
|A_{\Sigma_i}(p)|^2\leq K  |\mathbf{x}(p)|^2.
$$
That is, 
$$
\sup_{\Sigma_i\cap B_{2R}} |A_{\Sigma_i}|^2 \leq 4 K R^2.
$$	
Combining this with the area bound and the Arzel\`{a}-Ascoli theorem, gives that, up to passing to a subsequence, the $\Sigma_i$ converge, possibly with multiplicities, in $C^\infty_{loc}(B_{2R})$ to a limit $\Sigma^{\prime\prime}$. As $\Sigma^{\prime\prime}$ is a smooth solution to \eqref{ExpanderEqn} and there are no closed self-expanders, each component of $\Sigma^{\prime\prime}$ meets $\Sigma^{\prime}$.  In particular, as the $\Sigma_i$ converge with multiplicity one to $\Sigma^\prime$ in $B_{2R}\setminus \bar{B}_R$, the $\Sigma_i$ converge to $\Sigma^{\prime\prime}$ with multiplicity one in $B_{2R}$. Hence, setting $\Sigma=\Sigma^\prime\cup \Sigma^{\prime\prime}$ one obtains a smooth asymptotically conical self-expander with $\Sigma_i\to \Sigma$ in $C^{\infty}_{loc}(\Real^{n+1})$. As each $\Sigma_i$ has positive mean curvature, $\Sigma$ has non-negative mean curvature. However, as $\sigma=\mathcal{L}(\Sigma)$ has $H_{\mathcal{L}(\Sigma)}\geq h_0>0$ and $L_{\Sigma} H_{\Sigma}=-H_\Sigma\leq 0$, the strong maximum principle implies $\Sigma$ has positive mean curvature completing the proof.
\end{proof}

\section{Properness of map $\Pi$} \label{ProperSec}
Before proving Theorems \ref{2DProperThm}, \ref{EntropyProperThm}  and \ref{MCProperThm} we need the following auxiliary proposition that relates sequential compactness in $\mathcal{ACE}^{k,\alpha}_n(\Gamma)$ to locally smooth compactness in $\Real^{n+1}$. 

\begin{prop} \label{ParametrizeProp}
For $\Gamma\in\mathcal{ACH}^{k,\alpha}_n$, if $\varphi\in\mathcal{V}_{\mathrm{emb}}^{k,\alpha}(\Gamma)$, and $[\mathbf{f}_i]\in \mathcal{ACE}^{k,\alpha}_n(\Gamma)$ satisfy:
\begin{enumerate}
\item \label{TraceConvergeItem} $\mathrm{tr}^1_\infty [\mathbf{f}_i] =\varphi_i\to \varphi$ in $C^{k,\alpha}(\mathcal{L}(\Gamma); \Real^{n+1})$;
\item \label{LocConvergeItem} $\mathbf{f}_i(\Gamma)=\Sigma_i\to \Sigma$ in $C^\infty_{loc}(\Real^{n+1})$ for some hypersurface $\Sigma$,
\end{enumerate}
then $\Sigma\in \mathcal{ACH}_n^{k,\alpha}$ and is a self-expander. Moreover, there is a parameterization $\mathbf{f}\colon \Gamma \to \Sigma\subset\mathbb{R}^{n+1}$ so that  
\begin{enumerate}
\item \label{ACExpanderItem} $[\mathbf{f}]\in\mathcal{ACE}_n^{k,\alpha}(\Gamma)$;
\item \label{PrescribeTraceItem} $\mathrm{tr}^1_\infty [\mathbf{f}]=\varphi$; and;
\item \label{ACEConvergeItem} $[\mathbf{f}_i]\to [\mathbf{f}]$ in the topology of $\mathcal{ACE}^{k,\alpha}_n(\Gamma)$.
\end{enumerate} 
\end{prop}

\begin{proof}
First observe that as each $\Sigma_i$ satisfies \eqref{ExpanderEqn}, the nature of the convergence ensures that $\Sigma$ does as well. Let 
$$
\mathcal{C}_i=\mathscr{E}^{\mathrm{H}}_1[\varphi_i](\mathcal{C}(\Gamma))=\mathcal{C}(\Sigma_i) \mbox{ and } \mathcal{C}=\mathscr{E}^{\mathrm{H}}_1[\varphi](\mathcal{C}(\Gamma)).
$$
Then $\mathcal{C}_i$ and $\mathcal{C}$ are $C^{k,\alpha}$-regular cones. By our hypothesis \eqref{TraceConvergeItem},
$$
\mathcal{C}_i\to\mathcal{C} \mbox{ in $C^{k,\alpha}_{loc}(\mathbb{R}^{n+1}\setminus\{\mathbf{0}\})$},
$$
which further implies 
$$
\mathcal{L}(\Sigma_i)\to\mathcal{L}[\mathcal{C}] \mbox{ in $C^{k,\alpha}(\mathbb{S}^n)$}.
$$
Thus, by Corollary \ref{EndCompactCor} and our hypothesis \eqref{LocConvergeItem}, we have $\Sigma\in\mathcal{ACH}^{k,\alpha}_n$ with $\mathcal{C}(\Sigma)=\mathcal{C}$.

As $\varphi\in \mathcal{V}_{\mathrm{emb}}^{k,\alpha}(\Gamma)$, the hypothesis \eqref{TraceConvergeItem} ensures that
$$
\mathscr{E}^{\mathrm{H}}_1[\varphi_i]\circ\mathscr{E}^{\mathrm{H}}_1[\varphi]^{-1}\to\mathbf{x}|_{\mathcal{C}} \mbox{ in $C^{k,\alpha}_{loc}(\mathcal{C};\mathbb{R}^{n+1})$}.
$$
Observe, that $\mathscr{E}^{\mathrm{H}}_1[\varphi_i]\circ\mathscr{E}^{\mathrm{H}}_1[\varphi]^{-1}$ are homogeneous of degree one, so we denote their traces at infinity by $\tilde{\varphi}_i$. Thus, letting $\mathcal{L}$ be the link of $\mathcal{C}$, 
$$
\tilde{\varphi}_i\to \mathbf{x}|_{\mathcal{L}} \mbox{ in $C^{k,\alpha}(\mathcal{L};\mathbb{R}^{n+1})$}.
$$
Let $\mathbf{h}_i\in C^{k,\alpha}_1\cap C_{1,\mathrm{H}}^{k}(\Sigma;\mathbb{R}^{n+1})$ be chosen so that 
$$
\mathscr{L}_{\Sigma}\mathbf{h}_i=\Delta_\Sigma\mathbf{h}_i+\frac{1}{2}\mathbf{x}\cdot\nabla_\Sigma\mathbf{h}_i-\frac{1}{2}\mathbf{h}_i=\mathbf{0} \mbox{ and } \mathrm{tr}_\infty^1[\mathbf{h}_i]=\tilde{\varphi}_i-\mathbf{x}|_{\mathcal{L}}.
$$
By \cite[Corollary 5.8]{BernsteinWangBanach} there is a unique such $\mathbf{h}_i$ and it satisfies the estimate 
$$
\Vert\mathbf{h}_i\Vert_{k,\alpha}^{(1)} \leq C^\prime \Vert\tilde{\varphi}_i-\mathbf{x}|_{\mathcal{L}}\Vert_{k,\alpha}
$$
for some constant $C^{\prime}=C^\prime(\Sigma,n,k,\alpha)$. We let 
$$
\mathbf{g}_i=\mathbf{x}|_\Sigma+\mathbf{h}_i \mbox{ and } \Upsilon_i=\mathbf{g}_i(\Sigma).
$$
It is clear that for $i$ sufficiently large, $\mathbf{g}_i\in\mathcal{ACH}^{k,\alpha}_n(\Sigma)$ and $\mathrm{tr}_\infty^1[\mathbf{g}_i]=\tilde{\varphi}_i$. Thus, by  \cite[Proposition 3.3]{BernsteinWangBanach}, $\Upsilon_i\in\mathcal{ACH}^{k,\alpha}_n$ and $\mathcal{C}(\Upsilon_i)=\mathcal{C}_i$.

Now pick a transverse section $\mathbf{v}$ on $\Sigma$ that, outside a compact set, equals $\mathscr{E}_{\mathbf{w}}[\mathbf{w}]=\mathbf{w}\circ\pi_{\mathbf{w}}|_\Sigma$ for $\mathbf{w}$ a chosen homogeneous transverse section on $\mathcal{C}$. For $i$ sufficiently large, $\mathbf{v}_i=\mathbf{v}\circ\mathbf{g}_i^{-1}$ is an asymptotically homogeneous, transverse section on $\Upsilon_i$. By Proposition \ref{AsympRegProp}, for $i$ large, $\Sigma_i$ lies in a $\mathbf{v}_i$-regular neighborhood of $\Upsilon_i$ and is transverse to $\mathscr{E}_{\mathbf{v}_i}[\mathbf{v}_i]=\mathbf{v}_i\circ\pi_{\mathbf{v}_i}$. In particular, $\pi_{\mathbf{v}_i}|_{\Sigma_i}\colon\Sigma_i\to\Upsilon_i$ is an element of $\mathcal{ACH}^{k,\alpha}_n(\Sigma_i)$. Thus, for $i$ large, there is a unique $u_i\in C^{k,\alpha}_1\cap C^{k}_{1,0}(\Sigma)$ so that $\Sigma_i$ can be parametrized by the map 
$$
\tilde{\mathbf{f}}_i=(\pi_{\mathbf{v}_i}|_{\Sigma_i})^{-1}\circ\mathbf{g}_i=\mathbf{g}_i+u_i\mathbf{v}
$$
which is an element of $\mathcal{ACH}^{k,\alpha}_n(\Sigma)$ by \cite[Proposition 3.3]{BernsteinWangBanach}. Moreover, $\Vert u_i \Vert_{k,\alpha}^{(1)}$ is uniformly bounded and $\Vert u_i \Vert_1^{(1)}\to 0$.

Observe, that there is a $\delta>0$ (independent of $i$) so that for $i$ sufficiently large
$$
\inf_{p\in\Sigma} |\mathbf{v}\cdot(\mathbf{n}_{\Sigma_i}\circ\tilde{\mathbf{f}}_i)|>\delta.
$$
As $\Sigma_i$ is a self-expander, it follows from \cite[Lemma 7.2]{BernsteinWangBanach} and direct calculations that
$$
\mathscr{L}_\Sigma u_i=-\frac{\mathbf{n}_{\Sigma_i}\circ\tilde{\mathbf{f}}_i}{\mathbf{v}\cdot (\mathbf{n}_{\Sigma_i}\circ\tilde{\mathbf{f}}_i)} \cdot\left(2\nabla_\Sigma u_i\cdot\nabla_\Sigma\mathbf{v}+u_i\left(\mathscr{L}_\Sigma+\frac{1}{2}\right)\mathbf{v}+(g_{\tilde{\mathbf{f}}_i}^{-1}-g_\Sigma^{-1})^{jl}(\nabla_\Sigma^2\tilde{\mathbf{f}}_i)_{jl}\right)
$$
where $g_{\tilde{\mathbf{f}}_i}$ and $g_\Sigma$ are the pull-back metrics of Euclidean one by $\tilde{\mathbf{f}}_i$ and $\mathbf{x}|_\Sigma$, respectively. One further uses \cite[Proposition 3.1]{BernsteinWangBanach} to see that, for $i$ large, the right hand side are elements of $C^{k-2,\alpha}_{-1}(\Sigma)$ with uniformly bounded $\Vert\cdot\Vert_{k-2,\alpha}^{(-1)}$ norm. Hence, by \cite[Theorem 5.7 and Corollary 5.8]{BernsteinWangBanach}, $u_i\in\mathcal{D}^{k,\alpha}(\Sigma)$ and $\Vert u_i\Vert^*_{k,\alpha}$ is uniformly bounded. Here 
$$
\mathcal{D}^{k,\alpha}(\Sigma)=\set{u\in C^{k,\alpha}_1\cap C^{k-1,\alpha}_0\cap C^{k-2,\alpha}_{-1}(\Sigma)\colon \mathbf{x}\cdot\nabla_\Sigma u\in C^{k-2,\alpha}_{-1}(\Sigma)}
$$
is a Banach space with norm
$$
\Vert u\Vert_{k,\alpha}^*=\Vert u\Vert_{k-2,\alpha}^{(-1)}+\sum_{k-1\leq i\leq k}\Vert\nabla^i_\Sigma u\Vert_{\alpha}^{(1-k)}+\Vert\mathbf{x}\cdot\nabla_\Sigma u\Vert_{k-2,\alpha}^{(-1)}.
$$
As $\mathcal{D}^{k,\alpha}(\Sigma)$ is compactly embedded in $C^{k-1,\alpha}_1(\Sigma)$, we have $\Vert u_i \Vert_{k-1,\alpha}^{(1)}\to 0$. Thus it follows from \cite[Lemma 7.5]{BernsteinWangBanach} that $\Vert u_i \Vert_{k,\alpha}^*\to 0$ and so $\Vert\tilde{\mathbf{f}}_i-\mathbf{x}|_\Sigma\Vert_{k,\alpha}^{(1)}\to 0$.

We pick a large integer $I$ so that $\tilde{\mathbf{f}}_I$ is well-defined as above. Choose a representative $\mathbf{f}_I$ of $[\mathbf{f}_I]$. We define $\mathbf{f}=\tilde{\mathbf{f}}_I^{-1}\circ\mathbf{f}_I$, and it is clear that $\mathbf{f}(\Gamma)=\Sigma$. Moreover, by  \cite[Proposition 3.3]{BernsteinWangBanach}, $\mathbf{f}\in\mathcal{ACH}^{k,\alpha}(\Gamma)$ and 
$$
\mathcal{C}[\mathbf{f}]=\mathcal{C}[\tilde{\mathbf{f}}_I]^{-1}\circ\mathcal{C}[\mathbf{f}_I]=(\mathscr{E}^{\mathrm{H}}_1[\varphi_I]\circ\mathscr{E}^{\mathrm{H}}_1[\varphi]^{-1})^{-1}\circ\mathscr{E}^{\mathrm{H}}_1[\varphi_I]=\mathscr{E}^{\mathrm{H}}_1[\varphi].
$$
Thus, $[\mathbf{f}]$ represents a class in $\mathcal{ACE}^{k,\alpha}_n(\Gamma)$ which has $\Pi([\mathbf{f}])=\mathrm{tr}^1_\infty[\mathbf{f}]=\varphi$.

Hence, to complete the proof it remains only to show that $[\mathbf{f}_i]\to[\mathbf{f}]$ in the topology of $\mathcal{ACE}^{k,\alpha}_n(\Gamma)$. Observe that $\tilde{\mathbf{f}}_i\circ\mathbf{f}(\Gamma)=\Sigma_i$, and invoking  \cite[Proposition 3.3]{BernsteinWangBanach} again, $\tilde{\mathbf{f}}_i\circ\mathbf{f}\in\mathcal{ACH}^{k,\alpha}_n(\Gamma)$ and
$$
\mathcal{C}[\tilde{\mathbf{f}}_i\circ\mathbf{f}]=\mathcal{C}[\tilde{\mathbf{f}}_i]\circ\mathcal{C}[\mathbf{f}]=(\mathscr{E}^{\mathrm{H}}_1[\varphi_i]\circ\mathscr{E}^{\mathrm{H}}_1[\varphi]^{-1})\circ\mathscr{E}^{\mathrm{H}}_1[\varphi]=\mathscr{E}^{\mathrm{H}}_1[\varphi_i].
$$
This gives that $\tilde{\mathbf{f}}_i\circ\mathbf{f}$ is an element of $[\mathbf{f}_i]$. Moreover, by  \cite[Proposition 3.1]{BernsteinWangBanach}, $\tilde{\mathbf{f}}_i\circ\mathbf{f}\to\mathbf{f}$ in $C^{k,\alpha}_1(\Gamma;\mathbb{R}^{n+1})$. Therefore, $[\mathbf{f}_i]\to[\mathbf{f}]$ in the topology of $\mathcal{ACE}^{k,\alpha}_n(\Gamma)$.
\end{proof}

The proofs of properness of $\Pi$ now follow easily.

\begin{proof}[Proof of Theorem \ref{2DProperThm}]
The result follows directly from Item (1) of Theorem \ref{SmoothCpctThm}, Proposition \ref{ParametrizeProp} and an elementary topology fact, Lemma \ref{TopCompactLem}.
\end{proof}

\begin{proof}[Proof of Theorem \ref{EntropyProperThm}]
First, by Lemma \ref{ContEntropyConeLem}, $\mathcal{V}_{\mathrm{ent}}(\Gamma,\Lambda)$ is open in $C^{k,\alpha}(\mathcal{L}(\Gamma);\mathbb{R}^{n+1})$. Next, by Lemma \ref{ExpanderConeEntropyLem}, 
$$
\lambda[\mathbf{f}(\Gamma)]=\lambda[\mathcal{C}[\sigma]] \mbox{ where $\sigma=\Pi([\mathbf{f}])(\mathcal{L}(\Gamma))$}.
$$
Thus, the continuity of $\Pi$ and Lemma \ref{ContEntropyConeLem} imply that $\mathcal{U}_{\mathrm{ent}}(\Gamma,\Lambda)$ is open in $\mathcal{ACE}^{k,\alpha}_n(\Gamma)$. If $\mathcal{Z}\subset \mathcal{V}(\Gamma,\Lambda)$ is compact, then it follows from Lemma \ref{ContEntropyConeLem} that there is a $\Lambda_0<\Lambda$ so that
$$
\lambda[\mathscr{E}^{\mathrm{H}}_1[\varphi](\mathcal{C}(\Gamma))]\leq \Lambda_0
$$
for all $\varphi \in\mathcal{Z}$. Hence, if \eqref{Assump1} holds for some $\Lambda<2$, then the last claim follows from Item \eqref{EntropyCpctItem} of Theorem \ref{SmoothCpctThm}, Proposition \ref{ParametrizeProp} and Lemma \ref{TopCompactLem}.
\end{proof}

\begin{proof}[Proof of Theorem \ref{MCProperThm}]
Items \eqref{OpenMCExpanderItem} and \eqref{OpenMCConeItem} are straightforward. For $[\mathbf{f}]\in \mathcal{U}_{\mathrm{mc}}(\Gamma)$ let $\Sigma=\mathbf{f}(\Gamma)$. As $H_{\Sigma}>0$ and $L_{\Sigma} H_\Sigma=-H_{\Sigma}<0$, $L_{\Sigma}$ has a positive super-solution. Hence, it follows from the trick of Fischer-Colbrie and Schoen \cite{FCS} that $\Sigma$ is strictly stable. In particular, $\Sigma$ admits no non-trivial Jacobi fields and so Item \eqref{MCOpenItem} follows from \cite[Theorem 1.1 (4)]{BernsteinWangBanach}. Furthermore, if $\mathcal{Z}\subset \mathcal{V}_{\mathrm{mc}}(\Gamma)$ is compact, then there is an $h_0>0$ so $H_\sigma\geq h_0$, where $\sigma$ is the link of the cone $\mathscr{E}^\mathrm{H}_1[\varphi](\mathcal{C}(\Gamma))$, for all $\varphi\in\mathcal{Z}$. As such, Item \eqref{MCProperItem} follows from Item \eqref{MCCpctItem} of Theorem \ref{SmoothCpctThm}, Proposition \ref{ParametrizeProp} and Lemma \ref{TopCompactLem}. It remains only to show the final remark. Observe, that by Item \eqref{MCProperItem} and that $\mathcal{V}_\mathrm{mc}(\Gamma)$ is a compactly generated Hausdorff space, the map $\Pi|_{\mathcal{U}_\mathrm{mc}(\Gamma)}$ is a closed map. Hence, following the arguments in \cite[Proposition 4.46]{Lee}, the final remark is an immediate consequence of Items \eqref{MCOpenItem} and \eqref{MCProperItem}.
\end{proof}

\section{Existence of mean convex asymptotically conical self-expanders} \label{ApplicationSec}
We conclude by proving Corollary \ref{ApplicationCor}.  We first show existence and uniqueness of mean convex self-expanders asymptotic to rotationally symmetric cones.  This result is not new (see  \cite{AIC, Ding}), but we include a proof for the sake of completeness.
\begin{prop} \label{ApplicationProp}
For $n\geq 2$ let $\mathcal{C}\subset\mathbb{R}^{n+1}$ be a connected non-flat rotationally symmetric cone. There is a unique smooth self-expander $\Sigma$ that is smoothly asymptotic to $\mathcal{C}$. Moreover, $\Sigma$ satisfies $H_\Sigma>0$ and is an entire graph and so is diffeomorphic to $\mathbb{R}^n$ .
\end{prop}

\begin{proof}
Without loss of generality we assume that $\mathbf{e}_{n+1}$ is the axis of symmetry of $\mathcal{C}$ and $\mathcal{C}$ lies in the half-space $\set{x_{n+1}\geq 0}$. First, we show the existence of a strictly mean convex self-expanders asymptotic to $\mathcal{C}$. Clearly, there is a $\tau>0$ so $\mathcal{C}$ is the graph of a Lipschitz function $u_0\colon\Real^n\to \Real$ given by $u_0(x)=\tau|x|$. Let $u_0^{i}(x)=\tau \sqrt{i^{-1} +|x|^2}$ so $u_0^i\colon \Real^n\to \Real$ is a sequence of smooth functions which are strictly convex, have Lipschitz constant at most $\tau$ and are asymptotic to $u_0$. Moreover, the $u_0^i$ converge to $u_0$ uniformly. Let $\Sigma^i\subset\mathbb{R}^{n+1}$ be the graphs of the $u_0^i$. These are smooth convex hypersurfaces that are asymptotic to $\mathcal{C}$. Thus, Ecker-Huisken \cite{EHAnn} and the strong maximum principle applied to the evolution equation for mean curvature \cite{EHInvent} imply that there is a unique strictly mean convex, graphical solution, $\mathcal{M}^i=\set{\Sigma^i_t}_{t\geq 0}$ to mean curvature flow starting from $\Sigma^i$. Moreover, if we denote by $u^i(t,\cdot)$ the functions whose graphs are the $\Sigma_t^i$, then the $u^i(t, \cdot)$ have uniformly bounded Lipschitz constant. 

Hence, by Brakke's compactness (cf. \cite[\S 7]{IlmanenElliptic}), the $\mathcal{M}^i$ converges as Brakke flows to $\mathcal{M}=\set{\Sigma_t}_{t\geq 0}$ with $\Sigma_0=\mathcal{C}$. Furthermore, the interior estimates for graphical mean curvature flow \cite{EHInvent} and the Arzel\`{a}-Ascoli theorem imply that for $t>0$, $\Sigma_t$ has nonnegative mean curvature and is given by the graph of a smooth function $u(t,\cdot)$ on $\mathbb{R}^n$ with a uniform Lipschitz constant. Moreover, as $t\to 0^+$, $\Sigma_t\to\mathcal{C}$ in $C_{loc}^\infty(\mathbb{R}^{n+1}\setminus\set{\mathbf{0}})$. By the parabolic maximum principle, such a solution $u(t,\cdot)$ for $t>0$ is unique in the class of functions with at most linear growth. A consequence of this is that $u(t,x)=\sqrt{t}\, u(1,\frac{|x|}{\sqrt{t}})$. As such, $\Sigma=\Sigma_1$ is a graphical self-expanding hypersurface of revolution that is smoothly asymptotic to $\mathcal{C}$ and has nonnegative mean curvature. Furthermore, as $L_\Sigma H_\Sigma=-H_\Sigma\leq 0$ and $H_\mathcal{C}>0$, the strong maximum principle implies $H_\Sigma>0$.

Next we show the uniqueness of self-expanders asymptotic to the given cone $\mathcal{C}$. Suppose $\Sigma^\prime$ is another smooth self-expander smoothly asymptotic to $\mathcal{C}$. Let 
$$
s_0=\inf\set{s>0 \colon \Sigma^\prime\cap (\Sigma+s\mathbf{e}_{n+1})=\emptyset} \geq 0.
$$
We claim $s_0=0$ and so $\Sigma^\prime$ lies below $\Sigma$. Suppose not, i.e., $s_0>0$. By Item \eqref{LinearDecayItem} of Proposition \ref{AsympRegProp}, $\Sigma+s_0\mathbf{e}_{n+1}$ touches $\Sigma^\prime$ from above at some interior point. However, we observe that $H+\frac{1}{2}\mathbf{x}\cdot \mathbf{n}>0$ on $\Sigma+s_0\mathbf{e}_{n+1}$ where $\mathbf{n}$ is chosen to be upward. This leads to a contradiction with the strong maximum principle. By similar arguments, $\Sigma^\prime$ also lies above $\Sigma$. Thus $\Sigma^\prime=\Sigma$ proving the uniqueness.
\end{proof}

\begin{proof}[Proof of Corollary \ref{ApplicationCor}]
Let $\set{\sigma_t}_{t\in [0, T)}$ be the solution of mean curvature flow in $\mathbb{S}^{n}$ with $\sigma_0=\sigma$ and $T$ the maximal time of existence.  Pick a parameterization $\varphi_0\colon \sigma\to \sigma_{0}$ and choose $\varphi_t\colon \sigma\to \sigma_t$ a corresponding evolution of parameterizations. As $\sigma_0$ is mean convex, the maximum principle implies that $\sigma_t$ is as well for all $t\in (0,T)$. When $n=2$, results of Grayson \cite{Grayson} and Zhu \cite[Corollary 4.2]{Zhu} give that $\sigma_t$ smoothly shrinks to a round point in finite time. When $n\geq 3$, by \cite{HuiskenMCFSphere}, the condition \eqref{PinchingEqn} is preserved, which, together with the mean convexity and connectedness, implies that $T<\infty$ and the flow disappears in a round point at time $T$. 

Choose a rotationally symmetric hypersurface $\tilde{\sigma}\subset\mathbb{S}^n$ and a $T_0$ sufficiently close to $T$, so $\sigma_{T_0}$ is a small normal graph over $\tilde{\sigma}$. Thus, there is a path of smooth, strictly mean convex embeddings, $\psi_t\colon\tilde{\sigma}\to\mathbb{S}^n$,  $t\in [0,1]$, so $\psi_0(\tilde{\sigma})=\sigma_{T_0}$ and $\psi_1=\mathbf{x}|_{\tilde{\sigma}}$. Thus, setting
$$
\tilde{\varphi}_t=\left\{
\begin{array}{cc}
\varphi_t\circ\varphi_{T_0}^{-1}\circ\psi_0 & 0\leq t\leq T_0 \\
\psi_{t-T_0} & T_0 \leq t \leq T_0+1
\end{array}
\right.
$$
one obtains a path of strictly mean convex $C^{k,\alpha}$-embeddings of $\tilde{\sigma}$ into $\mathbb{S}^n$ connecting $\tilde{\varphi}_{T_0+1}=\mathbf{x}|_{\tilde{\sigma}}$ to $\tilde{\varphi}_0\colon\tilde{\sigma}\to\sigma$. By Proposition \ref{ApplicationProp} there is a unique smooth self-expander $\Gamma$ that is smoothly asymptotic to $\mathcal{C}[\tilde{\sigma}]$ and $H_\Gamma>0$. Hence, Theorem \ref{MCProperThm} implies there is an $[\mathbf{f}_0]\in\mathcal{U}_\mathrm{mc}(\Gamma)$ with $\Pi([\mathbf{f}_0])=\tilde{\varphi}_0$ and so $\Sigma=\mathbf{f}_0(\Gamma)$ is the desired element.

In what follows we restrict our discussions to those $n$ such that $\pi_0(\mathrm{Diff}^+(\mathbb{S}^{n-1}))=0$. For $n\geq 3$ we let
$$
\mathcal{V}=\set{\varphi\in\mathcal{V}_{\mathrm{mc}}(\Gamma)\colon \mbox{$\mathcal{L}[\mathscr{E}^\mathrm{H}_1[\varphi](\mathcal{C}(\Gamma))]$ satisfies \eqref{PinchingEqn}}},
$$ 
which is an open subset of $\mathcal{V}_{\mathrm{mc}}(\Gamma)$; for $n=2$ let $\mathcal{V}=\mathcal{V}_\mathrm{mc}(\Gamma)$. By the previous discussions and the hypothesis that $\mathrm{Diff}^+(\mathcal{L}(\Gamma))\simeq\mathrm{Diff}^+(\mathbb{S}^{n-1})$ is path-connected, we observe $\mathcal{V}$ has exactly two components. Let $\mathcal{V}_+$ be the component of $\mathcal{V}$ that contains $\mathbf{x}|_{\mathcal{L}(\Gamma)}$ and let $\mathcal{V}_-$ be the component containing $\mathbf{x}|_{\mathcal{L}(\Gamma)}\circ I$ where $I\in\mathrm{Diff}(\mathcal{L}(\Gamma))$ is an orientation-reversing involution. If $\mathcal{U}_\pm=\mathcal{U}_\mathrm{mc}(\Gamma)\cap \Pi^{-1}(\mathcal{V}_\pm)$,  then it follows from Theorem \ref{MCProperThm} and the uniqueness of $\Gamma$ that $\Pi|_{\mathcal{U}_\pm}\colon\mathcal{U}_\pm\to\mathcal{V}_\pm$ is a diffeomorphism.

Now suppose $\varphi,\psi\in\mathcal{V}_+$ so that $\mathscr{E}^\mathrm{H}_1[\varphi](\mathcal{C}(\Gamma))=\mathscr{E}^\mathrm{H}_1[\psi](\mathcal{C}(\Gamma))$. If $[\mathbf{f}], [\mathbf{g}]\in\mathcal{U}_+$ with $\Pi([\mathbf{f}])=\varphi$ and $\Pi([\mathbf{g}])=\psi$, then we will show $\mathbf{f}(\Gamma)=\mathbf{g}(\Gamma)$. Observe that there is a path of homogeneous $\Phi_t\in\mathrm{Diff}^+(\mathcal{C}(\Gamma))$ with $\Phi_0=\mathscr{E}^\mathrm{H}_1[\varphi]^{-1}\circ\mathscr{E}^\mathrm{H}_1[\psi]$ and $\Phi_1=\mathbf{x}|_{\mathcal{C}(\Gamma)}$. Thus, Theorem \ref{MCProperThm} and the uniqueness of $\Gamma$ imply that there is a path of $[\mathbf{h}_t]\in\mathcal{ACE}^{k,\alpha}_n(\Gamma)$ so that $\mathbf{h}_t(\Gamma)=\Gamma$ and $\Pi([\mathbf{h}_t])=\mathrm{tr}^1_\infty[\Phi_t]$. In particular, $[\mathbf{f}\circ\mathbf{h}_0]\in\mathcal{U}_+$ and $\Pi([\mathbf{f}\circ\mathbf{h}_0])=\psi$. As $\Pi|_{\mathcal{U}_+}$ is a diffeomorphism, $[\mathbf{g}]=[\mathbf{f}\circ\mathbf{h}_0]$ and so $\mathbf{g}(\Gamma)=\mathbf{f}(\Gamma)$. By symmetries, the same claim holds on $\mathcal{V}_-$. Finally, observe that as $\Gamma$ is diffeomorphic to $\mathbb{R}^n$ there is an orientation-reversing diffeomorphism $\tilde{I}\in\mathcal{ACH}^{k,\alpha}_n(\Gamma)$ so that $\mathrm{tr}_\infty^1[\tilde{I}]=I$. Moreover, $[\mathbf{f}]\in\mathcal{U}_+$ if and only if $[\mathbf{f}\circ\tilde{I}]\in\mathcal{U}_-$. This concludes the proof of uniqueness.
 
\end{proof}

\appendix

\section{}

\begin{lem} \label{TopCompactLem}
Let $X$ be a topological space. Suppose $X$ has a countable cover, $\set{A_i}_{i\in\mathbb{N}}$, of closed subsets so that each $A_i$ is metrizable. If $K$ is a sequentially compact subspace of $X$, then it is compact. 
\end{lem}

\begin{proof}
Let $\set{U_\alpha}$ be an arbitrary collection of open sets of $X$ that covers $K$. As $K$ is sequentially compact, so is every $K\cap A_i$. Since $A_i$ is metrizable, there is a finite subcollection of $\set{U_\alpha}$ that covers $K\cap A_i$. Thus, there is a countable subcollection $\set{U_{\alpha_i}}_{i\in\mathbb{N}}\subset\set{U_\alpha}$ that covers $K$. We show that there is a finite subcollection of $\set{U_{\alpha_i}}_{i\in\mathbb{N}}$ that covers $K$, implying the compactness of $K$. We argue by contradiction. Suppose not, then pick $x_j\in K\setminus (\bigcup_{i=1}^j U_{\alpha_i})$. Thus, up to passing to a subsequence, $x_j\to x\in K$. However, $x\in U_{\alpha_{i_0}}$ for some $i_0$, so for $j$ sufficiently large $x_j\in U_{\alpha_{i_0}}$, which leads to a contradiction.
\end{proof}

\end{document}